\documentclass[reqno]{amsart}
\usepackage{float}
\usepackage{amsmath}
\usepackage{graphicx}
\usepackage{latexsym}
\usepackage{amsfonts}
\usepackage{xcolor}
\usepackage{amssymb}
\usepackage[round]{natbib}
\usepackage{lipsum}
\usepackage{hyperref}
\usepackage[utf8]{inputenc}

\usepackage{appendix}

\newtheorem{theorem}{Theorem}
\newtheorem{lemma}{Lemma}
\newtheorem{proposition}{Proposition}

\newtheorem{corollary}{Corollary}
\newtheorem{remark}{Remark}

\newcommand{\ra}{{\rightarrow}}

\newcommand{\R}{{\mathbb R}} 
 
\newcommand{\N}{{\mathbb N}}  
\newcommand{\p}{{\mathbf P}}
\newcommand{\pr}{{\mathbf P}} 
\newcommand{\Z}{{\mathbb Z}}

\newcommand{\E}{{\mathbf E}}

\newcommand{\bi}{\begin{itemize}}
\newcommand{\ei}{\end{itemize}}

		\definecolor{cadmiumgreen}{rgb}{0.0, 0.42, 0.24}

\begin{document}

\title[Invariance principles]{Invariance Principles for Integrated Random Walks Conditioned to Stay Positive}

\author[Duraj]{Jetlir Duraj}
\address{Department of Economics, University of Pittsburgh, USA }
\email{jetlirduraj@gmail.com}

\author[B\"ar]{Michael B\"ar}
\address{msg systems ag}
\email{michael.baer@msg.group}

\author[Wachtel]{Vitali Wachtel}
\address{Institut f\"ur Mathematik, Universit\"at Augsburg, 86135 Augsburg, Germany}
\email{vitali.wachtel@mathematik.uni-augsburg.de}

\begin{abstract}
Let $S(n)$ be a centered random walk with finite second moment. 
We consider the integrated random walk $T(n)=S(0)+S(1)+\ldots+S(n)$. We prove invariance principles for the meander and for the bridge of this process, under the condition that the integrated random walk remains positive. Furthermore, we prove the functional convergence of its Doob's $h$-transform to the $h$-transform of the Kolmogorov diffusion conditioned to stay positive. 
\end{abstract}

\maketitle
\section{Introduction, main results and discussion}




Integrated random walks conditioned to stay positive have become a popular topic in probability. There are many applications of integrated random walks in models in mathematical physics and beyond, but also within probability theory. For example, \citet{caravennadeuschel} consider models of random fields with pinning and wetting, and the integrated random walk appears naturally as a tool in the study of the so-called `free' case. The survey articles \citet{aurzadasimon} and \citet{majumdar} contain several other applications. 

While there are many papers that have studied integrated random walks under the condition of positivity, until now they have focused either on special
cases of distributions of increments
or on characterizing asymptotic behavior of the exit time (e.g. \citet{aurzada}, \citet{dembo}, \citet{dw12}, \citet{gao}, \citet{sinai}, \citet{v1} and \citet{v2}). In this paper we prove invariance principles for the meander and for the bridge of the integrated random walk conditioned to stay positive. Furthermore, using the harmonic function for the integrated random walk conditioned to stay positive constructed in \citet{dw12}, we prove a functional convergence result for the Doob-$h$-transforms of the integrated random walks conditioned to stay positive.
\\\indent The `price' we have to pay for the proofs is existence of 
a $(2+\delta)$-moment of 
the step-distribution of the random walk. Just as for the case of random walks in cones treated in \citet{invariance}, this allows the use of the strong approximation of the random walk through the Brownian motion, for the mass of probability on the random walk paths in which the positivity condition for the integrated random walk is far from being violated. 

Hence, the proofs are an adaptation of the strategy in \citet{rwcones} and \citet{invariance}. In order to apply their strategy, we focus not on the integrated random walk (which is in general not  a Markov chain), but on the \emph{pair} integrated random walk \emph{and} random walk. 
This is a Markov chain and the condition of positivity for the integrated random walk can be formulated as a geometric condition of this chain not leaving $\R_+\times\R$. 

Our proof strategy requires establishing several weak convergence results for the respective continuous limits of the above mentioned Markov chain, namely for the Kolmogorov diffusion. The Kolmogorov diffusion is the \emph{pair} consisting of the integral of the Brownian motion and the Brownian motion itself. It is a Markov process in continuous time. The limit process conditional on posivitity of the first coordinate has been studied extensively in \citet{groeneboom}. What is needed for the invariance principle for random walks is the statement that the Kolmogorov diffusion, started at some point in the non-negative quadrant, converges weakly when the starting point tends
to zero within the non-negative quadrant. A similar result is needed for the Doob-$h$-transforms of the Kolmogorov diffusion. \citet{groeneboom} proves a special case of this result but we need a much more general statement for the invariance principles. As an auxiliary result, we prove in the last part of the paper some sharp estimates and asymptotics for the density of the Kolmogorov diffusion killed when the first coordinate becomes non-positive.


Once the invariance principle for the centered random walk is proven, certain invariance principles for 
integrated walks with drift under modified conditioning follow in a straightforward way. We note these down after proving the main results.

\subsection{Assumptions on the random walk}\label{subsubsec:assumptionsint}
Let $\{X_i\}_{i\in\N}$ be i.i.d. random variables with $\E[X_i] = 0, \E[X_i^2] = 1$ and $\E[|X_i|^{2+\delta}]<\infty$ for some $\delta>0$ given and fixed throughout. For every starting point $(x,y)$ define
\[
S(n) = y+ X_1+\dots+X_n, \quad n\in \N_0,
\]
and 
\begin{align*}
T(n) &= x+S(1)+\dots+S(n)\\
&= x+ny+nX_1+(n-1)X_2+\dots+X_n,\quad n\in \N_0.
\end{align*}
We define the Markov chain $Z(n) = (T(n),S(n)),n\in \N_0$. We can write the law of motion of the Markov chain $Z$ explicitly as follows:
\begin{equation}
    \label{eq:dynamicint}
    Z(n+1) = 
    \begin{pmatrix}
    1 & 1\\
    0 & 1
    \end{pmatrix} Z(n) + X_{n+1}\begin{pmatrix}
    1\\
    1
    \end{pmatrix},\quad n\ge 0.
\end{equation}
It is easy to establish that, a two-dimensional Markov chain started at some $Z_0\in \R^2$ has as components an integrated random walk and the corresponding random walk, if and only if it satisfies the law of motion as in \eqref{eq:dynamicint}. 

We want to study the process $\{Z(n),n\ge 0\}$ under the condition that $T(n)$ remains positive. Therefore, the following stopping time plays a crucial role:
\begin{align*}
\tau_{(x,y)}=\inf\{n\geq 0: T(n)\le 0\}=\inf\{n \ge0: Z(n)\notin\R_+\times\R\}.
\end{align*}

\citet{dembo}, \citet{gao} and \citet{sinai} prove estimates for the probability $\pr(\tau_{(x,y)}>n)$, whereas \citet{v1} and \citet{v2} finds the asymptotic behavior of $\pr(\tau_{(x,y)}>n),n\ra\infty$ under specific assumptions on the random walk. 
\citet{dw12} have determined the asymptotic behavior of $\pr(\tau_{(x,y)}>n)$ for a large class of random walks with zero drift, proven integral and local limit theorems for $Z(n)$ on the event $\{\tau_{(x,y)}>n\}$. They also prove by construction the existence of a positive harmonic function $V$ for $Z(n)$ killed when leaving $\R_+\times\R$. 

\subsection{On the continuous time limits.}
The standard functional central limit theorem for random walks implies
immediately that the properly scaled chain $Z$ converges to the following functional of the Brownian motion.
Given $\{B_t,t\ge 0\}$ a standard Brownian motion, define the Kolmogorov diffusion started at $(x,y)$ as the process given by
\[
W_{(x,y)}(t) = (U_t,V_t)_{(x,y)} =\left(x+ty+\int_0^t B_sds,\ y+ B_t\right),\quad t\geq 0.
\]
Denote the exit time from $\R_+\times\R$ by $\tau^{bm}_{(x,y)}$. It is given as follows:
\[
\tau^{bm}_{(x,y)} = \inf\{t\geq 0: W_{(x,y)}(t)\notin \R_+\times\R\}.
\]
The generator of $W$ is given by the operator $\mathcal{A} = y\frac{\partial}{\partial x}+\frac{1}{2}\frac{\partial^2}{\partial y^2}$.
It is well-known that there exists a function $h:\R_+\times\R\longrightarrow \R$ with  $\mathcal{A}h=0$ on $\R_+\times\R, h(0,\cdot)\equiv 0$, 
This is the so-called harmonicity property of $h$ and it can also be written as
\[
\E_{(x,y)}[h(W(t)),\tau^{bm}>t] = h(x,y), \quad x>0,y\in \R, t>0.
\]
In other words, $h(W(t))\textbf{1}_{\{\tau^{bm}>t\}}$ is a nonnegative martingale. $h$ has both an integral representation (see appendix) and a representation in terms of confluent hypergeometric functions (see \citet{groeneboom}, or appendix). 
 It has the scaling property $h(x,y) = x^{\frac{1}{6}}h(1,x^{-\frac{1}{3}}y), x>0, y\in \R$.

We call the measure on $C[0,1]$ given by the conditional measure $W_{(x,y)}(\cdot\textrm{ }|\tau^{bm}>1)$ the \emph{Kolmogorov meander of length one} started at $(x,y)\in \R_+\times\R_+$.
 
The invariance principle for the meander of $Z$ requires us first to construct the \emph{Kolmogorov meander of length one started at $(0,0)$}.
This is given in the following Proposition. To state it we define $\alpha(x,y) = \max\{|x|^{\frac{1}{3}},|y|\}$, $x,y\in\R$.
\begin{proposition}
 \label{thm:Kolmomeander}
If $z_n\searrow (0,0)$ and $
    \log\left(\frac{1}{\alpha(z_n)^{\frac{1}{2}}}\right) = O(\alpha(z_n)^{-q_0})$ 
    for $q_0=\frac{\log(9/8)}{\log 2}$ then
\begin{equation}
     \label{eq:convintbm}
     W_{z_n}(\cdot\textrm{ }|\tau^{bm}>1)\textrm{ converges weakly in }C[0,\infty).
 \end{equation} 
 \end{proposition}

 We denote the limit process from \ref{thm:Kolmomeander} as $W_{(0,0)}(\cdot\textrm{ }|\tau^{bm}>1)$. Its respective distribution is denoted by $\pr_{(0,0)}(W\in\cdot|\tau^{bm}>1)$. 
 One can characterize the density explicitly (see appendix). 
 
 We can define 
  the $h$-transform of $W$ as follows: For any $t>0$ and each continuous and bounded functional $f_t:C[0,t]\ra\R$,

  \[
  \E^{(h)}_{(x,y)}[f_t(W)] = \E_{(x,y)}\left[f_t(W)\frac{h(W(t))}{h(x,y)},\tau^{bm}>t\right],\quad x>0, y\in\R.
  \]
  Finally, with the help of \eqref{eq:convintbm} we show the following convergence in $\left(C[0,1],||\cdot||_{\infty}\right)$.
\begin{equation}
    \label{eq:htransformKol}
    \p_{z_n}^{(h)}\quad \textrm{ converges weakly as }z_n\searrow (0,0),
\end{equation}
under some weak restrictions on $\{z_n:n\ge 1\}$. 

The proof of \eqref{eq:convintbm} and \eqref{eq:htransformKol} requires a careful study of the density 
\[
\bar p_t(x,y;u,v)dudv = \pr_{(x,y)}(W(t)\in (du,dv),\tau^{bm}>t),
\]
which is found in Section \ref{sec:kolmo}.

\subsection{Main results.} We state the invariance principles for the process $Z$.

Let $X^{(n)}(t) = \left(\frac{T([tn])}{n^{\frac{3}{2}}}, \frac{S([tn])}{n^{\frac{1}{2}}}\right)$ for $t\in [0,1]$ be the scaled discrete-time process. This is considered as a process in $D[0,1]$, which is equipped with the supremum norm $||\cdot||_{\infty}$. 
 
\begin{theorem}\label{thm:meander}
The process $X^{(n)}(t)$ for $t\in[0,1]$ conditioned on $\{\tau>n\}$ and started at some fixed $(x,y)\in \R_+\times\R$ converges weakly to the Kolmogorov meander of length one started at $(0,0)$. 


\end{theorem}

Time length of one has been chosen only due to convenience, as one can easily extend the results above to the following more general case by time scaling.
\begin{corollary}[Meander of length $t$]\label{thm:meandert}
The process $X^{(n)}(s)$ for $s\in[0,t]$ conditioned on $\{\tau>n\}$ and started at some fixed $(x,y)\in \R_+\times\R_+$ converges weakly to the Kolmogorov meander of length $t$ started at $(0,0)$.
\end{corollary}
Note that this result implies in particular the limit theorem for $X^{(n)}(1)$, which is already known from \citet{dw12}.


For the invariance principle for bridges we look at steps of random walk defined on a lattice of $\R$. Without any loss of generality, we assume for the next statements that $X_1$ is supported on $\Z$ and that the resulting random walk is \emph{aperiodic}. Aperiodicity with respect to $\Z$ means that for every $x\in \Z$, the smallest subgroup of $\Z$ which contains the set 
\[
\{y: y = x+z\textrm{ with some }z\textrm{ such that }\pr(X_1=z)>0\},
\]
is $\Z$ itself.

The necessary modifications for the general case of a lattice $a+b\Z$, with maximal span $b>0$ are clear. We can then state the invariance principle for the meander as follows. We denote by $a_n$ the scaling function 
$$
a_n(x,y) = (n^{-\frac{3}{2}}x,n^{-\frac{1}{2}}y), n\in \N,(x,y)\in\Z_+\times\Z.
$$

\begin{theorem}\label{thm:bridgeinv}
For each $x,u, y,v\in \Z_+$ the process $(X^{(n)}(\cdot)|X^{(n)}(1)=a_n(u,v),\tau>n)$ started at $a_n(x,y)$ converges weakly to a process $Y$. $Y$ is a continuous Markov process with $Y|_{D[0,t]}$ absolutely continuous w.r.t. $\p_{(0,0)}(W\in\cdot\textrm{ }|\tau^{bm}>t)$. Moreover, $Y_t\searrow 0,\textrm{ as } t\nearrow 1$ a.s. 
\end{theorem}

We also derive the density of $Y_t$ w.r.t. Lebesgue measure. We call the process $Y$ the \emph{positive Kolmogorov excursion (from zero and back)}. It deserves the name for obvious reasons: it is continuous, absolutely continuous w.r.t. meander measure of length one started at zero, and converges to zero as $t$ goes to one. The proof of the invariance principle for bridges follows in broad lines the strategy already used in \citet{caravenna} for proving invariance principle for bridges of random walks conditioned to stay positive. The broad steps are as follows: first prove weak convergence in $D[0,t]$, $t\in[0,1]$; then time-reverse the process and use convergence towards the meander to prove weak convergence for time $[t,1]$; this also incidentally proves tightness for the whole process. 

Analogously to the case of meander, the restriction to time length one is artificial so that the following corollary is straightforward.
\begin{corollary}[Bridge of length $t$]
For each $x,u, y,v\in \Z_+$ and $t>0$ the process $(X^{(n)(\cdot)}|X^{(n)}(1)=a_n(u,v),\tau>tn)$ started at $a_n(x,y)$ converges weakly to a process $Y^{(t)}$. $Y^{(t)}$ is a continuous Markov process with $Y_s^{(t)}$ absolutely continuous w.r.t. $\p_0(W_s\in\cdot\textrm{ }|\tau^{bm}>s)$. Moreover, $Y^{(t)}_s\searrow 0,\textrm{ as } s\nearrow t$ a.s. 
\end{corollary}

Our third invariance principle shows functional convergence of $V$-transforms for the integrated random walks conditioned to be positive. The limit process is the $h$-transform of the Kolmogorov diffusion. 

Formally, the $V$-transform of $Z$ started at $x,y>0$ is defined as 

\[
\pr_{(x,y)}^{(V)}(Z(n)\in d(u,v)) = \pr_{(x,y)}(Z(n)\in d(u,v),\tau>n)\frac{V(u,v)}{V(x,y)}.
\]

To prove the functional convergence of $V$-transforms, we use \eqref{eq:htransformKol} and an estimate for the difference between the harmonic function of the integrated random walk $V$ and that of the harmonic function for the Kolmogorov diffusion $h$. Due to the work in \cite{dw12}, this estimate is given when the step of the random walk has moments strictly above $\frac{5}{2}$.

\begin{theorem}
\label{thm:htransformintegrated}
Assume that $\E[|X|^{\frac{5}{2}+\epsilon_0}]<\infty$ for some $\epsilon_0>0$. Then for every fixed $(x,y)\in \R_+\times\R_+$ the process $X^{(n)}$ under $\pr^{(V)}$ converges weakly to the Kolmogorov diffusion under $\pr_{(0,0)}^{(h)}$. 
\end{theorem}

A final note regarding the choice of the space $D[0,1]$ is due. If we define $X^{(n)}(t) = \left(\frac{T(tn)}{n^{\frac{3}{2}}}, \frac{S(tn)}{n^{\frac{1}{2}}}\right)$ for $tn\in\{1,\dots,n\}$, and through linear interpolation for other $t\in [0,1]$, we can prove the invariance principles for the integrated random walk on $\left(C[0,1],||\cdot||_{\infty}\right)$ as well. This is because $\R_+\times\R$ is convex. The limit processes are the same and the proofs are similar.



\paragraph{\bf Notation.} In the following, for the rest of the paper, we denote by $c,C,C_{\epsilon},C_R$ etc, positive constants whose value may change from line to line and whose exact value is not important for the arguments.

\section{Proof of convergence towards Kolmogorov meander}

\subsection{Construction of the meander}

In this paragraph we prove Proposition \ref{thm:Kolmomeander}. We show that the sequence of processes with continuous paths in $\R^2$
\[
\{W(\cdot)|W(0) = z_n,\tau^{bm}>1\},
\]
converges as $z_n\searrow (0,0)$, $n\ra\infty$. 

\textbf{Convergence of finite-dimensional distributions.} We first introduce  the scaling property of the Kolmogorov diffusion, which will be used several times in the proof.

Define the map $a_t:\R^2\ra\R^2$ via $a_t(x,y) = (\frac{x}{t^{\frac{3}{2}}},\frac{y}{t^{\frac{1}{2}}})$ for $t>0$. Let $b_t = a_t^{-1}$. Then it holds
\begin{equation}
\label{eq:scalingK}
a_t\left(\left(W_s, s\le t, W(0) = (x,y)\right)\right) \overset{d}{=} \left(W_s, s\le 1, W(0) = a_t(x,y)\right)
\end{equation}
or equivalently
\begin{equation*}
\left(W_s, s\le t, W(0) = (x,y)\right) \overset{d}{=} b_t\left(\left(W_s, s\le 1, W(0) = a_t(x,y)\right)\right).
\end{equation*}
The scaling property implies
\begin{align*}
    \frac{\pr_{(x,y)}(\alpha (W(t))>R,\tau^{bm}>t)}{h(x,y)} = \frac{\pr_{(x,y)}(\alpha (b_t(W(1)))>R,\tau^{bm}>1)}{t^{\frac{1}{4}}h(a_t(x,y))}.
\end{align*}
It follows that

\begin{equation}
    \label{eq:helpk2}
    \lim_{R\ra\infty}\sup_{x,y>0,a_t(x,y)\le\frac{1}{2}}\frac{\pr_{(x,y)}(\alpha (W(t))>R,\tau^{bm}>t)}{h(x,y)} = 0.
\end{equation}

Note that
\begin{equation}
    \label{eq:helpk1}
    \lim_{\epsilon\ra 0}\sup_{(x,y)\in \R^2_+, \alpha(x,y)\le\epsilon^2}\left|\frac{\pr_{(x,y)}(\tau^{bm}>1)}{\varkappa h(x,y)}-1\right| = 0.
\end{equation}
This follows from Lemma 15 in \citet{dw12} and the scaling property of Kolmogorov diffusion.

Now we take in \eqref{eq:helpk1} $\epsilon_0>0$ such that 
\begin{equation}
    \label{eq:helpk3}
    \pr_{(x,y)}(\tau^{bm}>1)\ge \frac{1}{2}\varkappa h(x,y)\textrm{ for all }(x,y)\in \R_+\times\R,\alpha(x,y)\le \epsilon_0^2.
\end{equation}
Combining \eqref{eq:helpk3} with \eqref{eq:helpk2} delivers for every $t\in (0,1]$ that
\begin{equation}
    \label{eq:helpk4}
    \lim_{R\ra\infty}\sup_{x,y>0,a_t(x,y)\le \frac{1}{2}}\pr_{(x,y)}(\alpha(W(t))>R|\tau^{bm}>t) = 0.
\end{equation}



We can write for the density of the meander started at $(x,y)$
\begin{align*}
   m_t(x,y;u,v) &=  \frac{\bar p_t(x,y;u,v)\pr_{(u,v)}(\tau^{bm}>1-t)}{\pr_{(x,y)}(\tau^{bm}>1)}\\& = \frac{\bar p_t(x,y;u,v)}{h(x,y)}\frac{h(x,y)}{\pr_{(x,y)}(\tau^{bm}>1)}\pr_{(u,v)}(\tau^{bm}>1-t).
\end{align*}

We prove in the appendix \begin{equation}
\label{eq:densityKolconvergence}
\overline{p}_t(x,y;u,v) \sim h(x,y) \overline{h}(t,u,-v),\quad x,y\searrow 0
\end{equation}
uniformly in $(u,v)$ with $\alpha(u,v)\le R$ ($R>0$ arbitrary). Here $[0,\infty)\times\R_+\times\R\ni(t,u,v)\mapsto\bar h(t,u,-v)$ is also function defined in the appendix. Moreover, we prove the estimate
\begin{equation}
    \label{eq:needed}
    \sup_{x,y\in\R_+, \alpha(x,y)\le 1}\frac{\bar p_1(x,y;u,v)}{h(x,y)}\le \psi(u,v),
\end{equation}
for an integrable function $\psi:\R^2\ra\R$ such that $\psi h$ is integrable in $\R_+\times\R$.

\eqref{eq:densityKolconvergence} shows that the density $m_t(x,y;u,v)$ converges to $m_t(u,v)$, as $x,y\searrow 0$ uniformly for $\alpha(u,v)\le R$ and $R>0$ arbitrary. Note that \eqref{eq:needed} implies that $m_t(u,v)$ is integrable. Using this fact and \eqref{eq:helpk4} it follows that 

\[
\limsup_{x,y\searrow 0}\int\int \left|m_t(x,y;u,v)-m_t(u,v)\right|dudv = 0.
\]

It follows that $m_t(u,v)$ is also a density. Together with the Markov property, this implies the convergence of finite-dimensional distributions. 

\textbf{Tightness.} First, we note that due to the continuous mapping theorem, it is enough to focus on showing tightness of the second coordinate of $W$. That is, it is sufficient to show the tightness of Brownian motion killed whenever its integral (plus a start point) becomes non-positive.

We use Theorems 7.3, 7.4 in \citet{billingsley} and the Corollary of the latter result from the same reference. 

Fix $\epsilon,\delta>0$. For every $t\in[\delta,1)$ we have 
\begin{align}
\label{Bmeander.3}
\nonumber
&\pr_{z_n}\left(\sup_{u\in[t,t+\delta]}|B(u)-B(t)|>\epsilon\bigg|\tau^{bm}>1\right)\\
\nonumber
&\hspace{0.5cm} 
=\frac{\pr_{z_n}\left(\sup_{u\in[t,t+\delta]}|B(u)-B(t)|>\epsilon,\tau^{bm}>1\right)}{\pr_{z_n}(\tau^{bm}>1)}\\
\nonumber
&\hspace{0.5cm}\le 
\frac{\pr_{z_n}\left(\sup_{u\in[t,t+\delta]}|B(u)-B(t)|>\epsilon,\tau^{bm}>t\right)}{\pr_{z_n}(\tau^{bm}>1)}\\
\nonumber
&\hspace{0.5cm}= \frac{\pr_{z_n}(\tau^{bm}>t)}{\pr_{z_n}(\tau^{bm}>1)}\pr\left(\sup_{u\in[0,\delta]}|B(u)|>\epsilon\right)
   \le \frac{\pr_{z_n}(\tau^{bm}>\delta)}{\pr_{z_n}(\tau^{bm}>1)}\pr\left(\sup_{u\in[0,\delta]}|B(u)|>\epsilon\right)\\
&\hspace{0.5cm}\le \frac{\pr_{z_n}(\tau^{bm}>\delta)}{\pr_{z_n}(\tau^{bm}>1)}C_1 e^{-c_0\epsilon^2/\delta}
\le C_2\delta^{-1/4}e^{-c_0\epsilon^2/\delta}.
\end{align}
The second equality uses the Markov property. In the last two steps we have used 
the obvious bound
$$
\pr\left(\sup_{u\in[0,\delta]}|B(u)|>\epsilon\right)\le C_1 e^{-c_0\epsilon^2/\delta},
$$
and 
$\pr_{z_n}(\tau^{bm}>t)\sim\varkappa h(z_n)t^{-\frac{1}{4}}, n\ra\infty$ which follows from Lemma 15 in \citet{dw12}. Combining now \eqref{Bmeander.3} with the above-mentioned corollary to Theorem 7.4 in  \citet{billingsley},
we conclude tightness on $C[\delta,1]$ for every $\delta>0$.


In order to extend this property on the whole interval $[0,1]$ we need to show that the probability $\pr_{z_n}(\sup_{u\in [0,\delta]}|B(u)|>\epsilon|\tau^{bm}>1)$ can be made sufficiently small.

We first note that 
\begin{align*}
&\pr_{z_n}\left(\sup_{u\in[0,\delta]}|B(u)|>\epsilon\bigg|\tau^{bm}>1\right)
\le \pr_{z_n}\left(\sup_{u\in[0,\frac{\delta}{2}]}|B(u)|>\frac{3\epsilon}{4}\bigg|\tau^{bm}>1\right)\\
&\hspace{3cm} + \pr_{z_n}\left(\sup_{u\in[0,\frac{\delta}{2}]}|B(u)|\le\frac{3\epsilon}{4}, \sup_{u\in[0,\delta]}|B(u)|>\epsilon\bigg|\tau^{bm}>1\right).
\end{align*}
For the second summand we have 
\begin{align*}
&\pr_{z_n}\left(\sup_{u\in[0,\frac{\delta}{2}]}|B(u)|\le\frac{3\epsilon}{4}, \sup_{u\in[0,\delta]}|B(u)|>\epsilon\bigg|\tau^{bm}>1\right)\\
&\hspace{2cm}\le\frac{1}{\pr_{z_n}(\tau^{bm}>1)}\pr_{z_n}\left( \sup_{u\in[\delta/2,\delta]}|B(u)-B(\delta/2)|>\frac{\epsilon}{4}
,\tau^{bm}>\frac{\delta}{2}\right)\\
&\hspace{2cm}= \frac{\pr_{z_n}(\tau^{bm}>\frac{\delta}{2})}{\pr_{z_n}(\tau^{bm}>1)}
\pr\left(\sup_{u\in[0,\frac{\delta}{2}]}|B(u)|>\frac{\epsilon}{4}\right)\le \frac{\pr_{z_n}(\tau^{bm}>\frac{\delta}{2})}{\pr_{z_n}(\tau^{bm}>1)}C_1e^{-c_0\left(\frac{\epsilon}{4}\right)^2/\frac{\delta}{2}}.
\end{align*}
As a result, 
\begin{align*}
&\pr_{z_n}\left(\sup_{u\in[0,\delta]}|B(u)|>\epsilon\bigg|\tau^{bm}>1\right)\\
&\hspace{1cm}\le \pr_{z_n}\left(\sup_{u\in[0,\frac{\delta}{2}]}|B(u)|>\frac{3\epsilon}{4}\bigg|\tau^{bm}>1\right) +\frac{\pr_{z_n}(\tau^{bm}>\frac{\delta}{2})}{\pr_{z_n}(\tau^{bm}>1)}C_1e^{-c_0\epsilon^2/8\delta}.
\end{align*}
After $N$ iterations we arrive at the bound 
\begin{align}\label{eq:h2}
&\pr_{z_n}\left(\sup_{u\in[0,\delta]}|B(u)|>\epsilon\bigg|\tau^{bm}>1\right)\nonumber\\
&\hspace{0.3cm}\le \pr_{z_n}\left(\sup_{u\in[0,\frac{\delta}{2^N}]}|B(u)|>\frac{3^N\epsilon}{4^N}\bigg|\tau^{bm}>1\right)
+\sum_{k=1}^N\frac{\pr_{z_n}(\tau^{bm}>\frac{\delta}{2^k})}{\pr_{z_n}(\tau^{bm}>1)}
C_1e^{-c_0\frac{9^{k-1}\epsilon^2}{8^k\delta}}.
\end{align}
According to \eqref{eq:helpk1} it holds 
\begin{equation}
\label{eq:prop}
\sup_{\alpha(z)\le\frac{1}{2}}\frac{\pr_z(\tau^{bm}>1)}{h(z)}\le C_*.
\end{equation}

Let $N:=\max\{k\ge 1:2^{\frac{k}{2}}\le\frac{\sqrt{\delta}}{2\alpha(z_n)}\}$. Then it holds for $k\le N$ that $\alpha\left(a_{\frac{\delta}{2^k}}(z_n)\right) = \frac{2^{\frac{k}{2}}}{\sqrt{\delta}}\alpha(z_n)$. 
Then by the scaling property of Kolmogorov diffusion, \eqref{eq:prop} and Lemma 6 in \citet{dw12}, we have for $k\le N$ 
\begin{align}\label{eq:h3}
\pr_{z_n}\left(\tau^{bm}>\frac{\delta}{2^k}\right) = \pr_{a_{\frac{\delta}{2^k}}(z_n)}\left(\tau^{bm}>1\right)\le C_*\frac{2^{k/4}}{\delta^{\frac{1}{4}}}h(z_n).
\end{align}
Combining \eqref{eq:h2} and \eqref{eq:h3}, we obtain
\begin{align*}
\pr_{z_n}\left(\sup_{u\in[0,\delta]}|B(u)|>\epsilon\bigg|\tau^{bm}>1\right)
&\le \pr_{z_n}\left(\sup_{u\in[0,\frac{\delta}{2^N}]}|B(u)|>\frac{3^N\epsilon}{4^N}\right)/\pr_{z_n}(\tau^{bm}>1)\\
&\hspace{1cm}+ C\sum_{k=1}^N\frac{2^{k/4}}{\delta^{1/4}}e^{-c_09^{k-1}\epsilon^2/8^k\delta}\\
&\le \frac{Ce^{-c_09^N\epsilon^2/8^N\delta}}{h(z_n)}+C\delta^{-\frac{1}{4}}e^{-c_0\epsilon^2/8\delta}.
\end{align*}

By the definition of $N$ it holds $2^{\frac{N}{2}}>\frac{\sqrt{\delta}}{\alpha(z_n)}$. Therefore, 
\[
e^{-(9/8))^Nc_0\epsilon^2/\delta}\le e^{-c_1 \epsilon^2/(\delta^{1-\frac{q_0}{2}}\alpha(z_n)^{q_0})},
\]
where $q_0=\frac{\log(9/8)}{\log 2}\in(0,1)$.
Consequently,
$$
\pr_{z_n}\left(\sup_{u\in[0,\delta]}|B(u)|>\epsilon\bigg|\tau^{bm}>1\right)
\le C\left(\frac{e^{-c_1 \epsilon^2/(\delta^{1-\frac{q_0}{2}}\alpha(z_n)^{q_0})}}{h(z_n)}
+\delta^{-\frac{1}{4}}e^{-c_0\epsilon^2/8\delta}\right).
$$

 For every given $\epsilon>0$ one can choose $\delta$ so that the second summand on the right hand side becomes as small as needed. Clearly, the first summand also converges to zero, for an appropriate choice of $\delta$, if we assume that the sequence $z_n$ satisfies the condition
\begin{equation}
    \label{eq:mild}
    \log\left(\frac{1}{\alpha(z_n)^{\frac{1}{2}}}\right) = O(\alpha(z_n)^{-q_0}).
\end{equation}
\eqref{eq:mild} is easy to satisfy, e.g. take sequences $\{z_n,n\ge 1\}$ such that $\alpha(z_n)=O( n^{-\beta})$ for some $\beta>0$ as $n\ra\infty$. 

\subsection{Proof of Theorem \ref{thm:meandert}}



Let $X_n(t) = \left(\frac{T([tn])}{n^{\frac{3}{2}}},\frac{S([tn])}{n^{\frac{1}{2}}}\right)$ for $t\in[0,1]$ and let $\tau$ be the stopping time defined above. When we do not annotate or mention it explicitly, it is assumed the processes are started at zero. 

Let $f:D[0,1]\ra [0,1]$ be a bounded and uniformly continuous function. First, we introduce several helpful notations. For $k\leq n, z\in\R_+\times\R, w\in D[0,1]$ we define
\[
f(z,k,w) = f(z\textbf{1}_{\{t\leq \frac{k}{n}\}}+ w(t)\textbf{1}_{\{t> \frac{k}{n}\}}).
\]
The scaling function $a_n$ is defined by $a_n(z) = (\frac{z_1}{n^{\frac{3}{2}}},\frac{z_2}{n^{\frac{1}{2}}})$ and for $z = (x,y)$ we define \\$z^+= (x+n^{\frac{3}{2}-\gamma},y), z^-= (x-n^{\frac{3}{2}-\gamma,},y)$. We also define 
\[
K_{n,\epsilon} = \{z\in \R_+\times\R: y>0, x\geq n^{\frac{3}{2}-3\epsilon}\},
\]
for the set of values deep inside $\R_+\times\R$. 
Note that for $\epsilon>0$ small enough, $z^+,z^-\in K_{n,\epsilon'}$ if $z\in K_{n,\epsilon}$, where $\epsilon'>\epsilon$. Also let 
\[
\gamma_n =\inf\{k\geq 0: Z_k\in K_{n,\epsilon}\}
\]
be the corresponding entry time. This denotes the first time 
when $Z$ is deep inside $\R_+\times\R$. Intuitively speaking, because of the moment assumption, under the conditioning, the process $Z$ will spend most of the time in $K_{n,\epsilon}$ so that $\gamma_n$ will be relatively small. 
For some $\theta_n\ra 0,n\ra \infty$ and so that $\theta_n\sqrt{n}\ra \infty$ we use
$\alpha(z) = \max\{|x|^{\frac{1}{3}},|y|\}$ to define
\[
L_{n,\epsilon} = \{z\in K_{n,\epsilon}: \alpha(z)\leq \theta_n\sqrt{n}\}=\{z: 0<y\leq \theta_n\sqrt{n}, n^{\frac{3}{2}-3\epsilon}\leq x\leq \theta_n^3n^{\frac{3}{2}}\}.
\]
 Finally, we define for future reference the random variables $m_n: = \max_{k\leq n} \alpha(Z_k)$ and $M_n:=\max_{k\leq n}|S_k|$. \\\indent To show the invariance principle we have to show w.l.o.g.
\[
\E_{(x,y)}[f(X^{(n)})|\tau>n]\ra \E_{(0,0)}[f(W)|\tau^{bm}>1], \quad n\ra \infty.
\]
For the proof we will use strong approximation by Brownian motion, as stated in the following Lemma.\footnote{This is a classical result, see e.g. \citet{major}.} 
\begin{lemma}[Hungarian construction]\label{thm:coupling}
If $\E[X]^{2+\delta}<\infty$ for some $\delta\in (0,1),$ then for every $n\geq 1$ one can define a Brownian motion $B_t$ on the same probability space, so that for any $\gamma$ satisfying $0<\gamma<\frac{\delta}{2(2+\delta)}$ the event
\[
A_n = \{\sup_{u\leq n}|S_{[u]}-B_u|\leq n^{\frac{1}{2}-\gamma}\}
\]
satisfies 
\[
\p(A_n^c) = o(n^{-r}),
\]
with $r = r(\delta,\gamma) = -2\gamma-\gamma\delta+\frac{\delta}{2}$. 
\end{lemma}
As a first step, we see easily that for all $\epsilon>0$ small enough only the part $$\E[f(X_n),\gamma_n\leq n^{1-\epsilon}|\tau>n]$$  contributes something to the limit. This is because 
\[
\E_{(x,y)}[f(X^{(n)}),\gamma_n> n^{1-\epsilon}|\tau>n]\leq \frac{\p_{(x,y)}(\gamma_n>n^{1-\epsilon}, \tau>n^{1-\epsilon})}{\p_{(x,y)}(\tau>n)}\leq C(x,y)e^{-n^{\epsilon'}},
\]
with some suitable $\epsilon'>0$ small enough. This follows from Corollary 3 and Lemma 12 in \citet{dw12}.
\\As a second step, we prove 
\[
\frac{\p_{(x,y)}(\tau>n,\gamma_n\leq n^{1-\epsilon}, m_{\gamma_n}>\theta_n\sqrt{n})}{\p_{(x,y)}(\tau>n)} = o(1).
\]
Due to Markov property we have
\begin{align*}
&\p_{(x,y)}(\tau>n,\gamma_n\leq n^{1-\epsilon}, m_{\gamma_n}>\theta_n\sqrt{n})\\&\leq \sum_{k\leq n^{1-\epsilon}}\int_{K_{n,\epsilon}}\p_z(\tau>n-n^{1-\epsilon})\p(Z_{\gamma_n}\in dz, \tau>k=\gamma_n,m_{\gamma_n}>\theta_n\sqrt{n})\\
&\leq C\frac{1}{n^{\frac{1}{4}}}\E_{(x,y)}[h(Z_{\gamma_n}), \gamma_n\leq n^{1-\epsilon}\tau>\gamma_n, m_{\gamma_n}>\theta_n\sqrt{n}].
\end{align*}
Here we have used, that 
\[
\p_z(\tau>n)\leq C\frac{h(z)}{n^{\frac{1}{4}}},\quad \textrm{ uniformly in }z\in K_{n,\epsilon},\ n\geq 1.
\]
This inequality is contained in Lemma 18 in \citet{dw12}.
We now show that
\begin{equation}\label{eq:eqhelp1meander}
\E[h(Z_{\gamma_n}), \gamma_n<n^{1-\epsilon}\tau>\gamma_n, m_{\gamma_n}>\theta_n\sqrt{n}] = o(1).
\end{equation}
We use Lemma 6 in \citet{dw12} to see, that on $\{\gamma_n\leq n^{1-\epsilon}\}$ we have $h(Z_{\gamma_n})\leq Cm_{n^{1-\epsilon}}^\frac{1}{2}$. Therefore,
\[
\E_{(x,y)}[h(Z_{\gamma_n}), \gamma_n<n^{1-\epsilon}\tau>\gamma_n, m_{\gamma_n}>\theta_n\sqrt{n}]= O(\E_{(x,y)}[m_{n^{1-\epsilon}}^\frac{1}{2}, m_{n^{1-\epsilon}}>\theta_n\sqrt{n}]).
\]
Since
$$m_{n^{1-\epsilon}}\leq n^{\frac{1-\epsilon}{3}}M(n^{1-\epsilon})^{\frac{1}{3}}+ M(n^{1-\epsilon}),$$
we have 
\begin{align*}
\E_{(x,y)}[m_{n^{1-\epsilon}}^\frac{1}{2}, m_{n^{1-\epsilon}}>\theta_n\sqrt{n}]\leq & \E_{(x,y)}[n^{\frac{1-\epsilon}{6}}M(n^{1-\epsilon})^\frac{1}{6}, M(n^{1-\epsilon})>\theta_n^3 n^{\epsilon}\sqrt{n}]\\&+\E_{(x,y)}[M(n^{1-\epsilon})^\frac{1}{2}, M(n^{1-\epsilon})>\theta_n\sqrt{n}]\\
&\leq 2\E_{(x,y)}[M(n^{1-\epsilon})^\frac{1}{2}, M(n^{1-\epsilon})>\theta_n\sqrt{n}],
\end{align*}
for $\epsilon>0$ small enough and $\theta_n$ going to zero slowly enough.\footnote{E.g. $\theta_n^2 n^{\epsilon} \ra\infty$ is sufficient.}

Kolmogorov inequality gives 
\[
\E_{(x,y)}[M(n^{1-\epsilon})^\frac{1}{2}, M(n^{1-\epsilon})>\theta_n\sqrt{n}] = o(1),
\]
if $\theta_n$ goes to zero slowly enough.\footnote{$\theta_n^2 n^{\epsilon} \ra\infty$ is again sufficient.}
In all, the only part that contributes to the limit is 
\[
\frac{\E_{(x,y)}[f(X_n),\tau>n,\gamma_n\leq n^{1-\epsilon}, m_{\gamma_n}\leq \theta_n\sqrt{n}]}{\p_{(x,y)}(\tau>n)}.
\]
We note first that we have 
\begin{align*}
&\sup_{t\in [0,1]}\left|\left(a_n(Z(\gamma_n))\textbf{1}_{\{t\leq \frac{k}{n}\}}+X_n(t)\textbf{1}_{\{t>\frac{k}{n}\}}\right)-X_n(t)\right|\\
&\leq \max_{k\leq \gamma_n}|a_n(Z(k)-Z(\gamma_n))|\leq \frac{2M(\gamma_n)}{\sqrt{n}}\leq 2\theta_n
\end{align*}
on the set $\{m_{\gamma_n}\leq \theta_n\sqrt{n}\}$.
Let 
$$C_n = \{\gamma_n\leq n^{1-\epsilon}, m_{\gamma_n}\leq \theta_n\sqrt{n}\}\in\mathcal{F}_{\gamma_n}.$$ 
It follows from the uniform continuity of $f$ that 
\[
\E_{(x,y)}[f(X^{(n)}), C_n|\tau>n] = o(1) + \E_{(x,y)}[f(a_n(Z_{\gamma_n}),\gamma_n, X^{(n)}), C_n|\tau>n]. 
\]
Therefore, to prove the theorem, it suffices to consider convergence of 
\[
\E_{(x,y)}[f(a_n(Z_{\gamma_n}),\gamma_n, X^{(n)}), C_n|\tau>n].
\]
Note first that we can write with the Markov property
\begin{align*}
&\E_{(x,y)}[f(a_n(Z_{\gamma_n}),\gamma_n, X_n), C_n,\tau>n] \\&=\sum_{k\leq n^{1-\epsilon}}\int_{K_{n,\epsilon}}\p(\gamma_n = k, \tau>k, m_{\gamma_n}\leq \theta_n\sqrt{n},Z_k\in dz)\\&\times\E_z[f(a_n(z),k,X^{(n)})|\tau_z>n-k]\p_z(\tau>n-k).
\end{align*}
As a next step want to show
\begin{equation}\label{eq:help2meander}
\E_z[f(a_n(z),k,X^{(n)})|\tau_z>n-k] = (1+o(1))\E_{(0,0)}[f(W)|\tau^{bm}>1],
\end{equation}
uniformly for $z\in m_{n,\epsilon}, k\leq n^{1-\epsilon}$.

With \eqref{eq:help2meander}, it would follow
\[
\E_{(x,y)}[f(a_n(Z_{\gamma_n}),\gamma_n, X^{(n)}), C_n|\tau>n] = (1+o(1))\E_{(0,0)}[f(W)|\tau^{bm}>1]\frac{\p_{(x,y)}(\tau>n, C_n)}{\p_{(x,y)}(\tau>n)},
\]
and we know from calculations in subsection 3.3 in \citet{dw12}, that
\[
\frac{\p_{(x,y)}(\tau>n, C_n)}{\p_{(x,y)}(\tau>n)} = 1+o(1).
\]
This would finish the proof. Therefore, the rest of this section focuses on proving \eqref{eq:help2meander}.

According to Lemma \ref{thm:coupling}, one can define a walk $S(n)$ and a Brownian motion $B(u)$ in a common probability so that for the event 
\[
A_n = \{\sup_{u\leq n}|S([u])-B(u)|\leq n^{\frac{1}{2}-\gamma}\}.
\]
it holds
\[
\p(A_n^c) = o(n^{-r}),
\]
with $r = r(\delta,\gamma) = -2\gamma-\gamma\delta+\frac{\delta}{2}$.

First, we note, that
\begin{align*}
\E_z[f(a_n(z),k,X^{(n)})|\tau_z>n-k] &= \E_z[f(a_n(z),k,X^{(n)}),A_n|\tau_z>n-k] \\&+ \E_z[f(a_n(z),k,X^{(n)}),A_n^c|\tau_z>n-k]\\&=\E_z[f(a_n(z),k,X^{(n)}),A_n|\tau_z>n-k] +o(1),
\end{align*}
for all $\epsilon>0$ sufficiently small. This is because of $\p(A_n^c) = o(n^{-r})$  and the fact that uniformly for all $z\in {L}_{n,\epsilon}, k\leq n^{1-\epsilon}: n^{-r} = o(\p(\tau_z>n-k))$. The latter follows from the proof of Lemma 18 in \citet{dw12}.

Define 
$$
W^{(n)}(t) = \left(\frac{\int_0^{[nt]} B_sds}{n^{\frac{3}{2}}},\frac{B_{[nt]}}{n^{\frac{1}{2}}}\right).
$$
Due to uniform continuity of $f$ we have in $A_n$, uniformly for all 
$k\leq n^{1-\epsilon}, z\in K_{n,\epsilon}$ that 
\begin{equation}\label{eq:help3meander}
f(a_n(z),k,X^{(n)}) = f(a_n(z),k,W^{(n)}) +o(1).
\end{equation}
We know that 
\begin{equation}\label{eq:help4meander}
|a_n(z)-a_n(z^{\pm})|\leq O(n^{-\gamma}),
\end{equation}
uniformly in $z\in K_{n,\epsilon}$. 
One also has 
\begin{equation}\label{eq:help5meander}
\p_{z^{\pm}}(\tau^{bm}>n) = \chi h(z)n^{-\frac{1}{4}}(1+o(1)) = \p_z(\tau>n)
(1+o(1)),
\end{equation}
uniformly in $k\leq n^{1-\epsilon}, z\in K_{n,\epsilon}.$
These follow from considerations in subsection 3.2 of \citet{dw12}. \\Using \eqref{eq:help3meander},\eqref{eq:help4meander} and \eqref{eq:help5meander} and the scaling property of Brownian motion, it is easy to see that
\[
\E_z[f(a_n(z),k,X^{(n)}), A_n|\tau_z>n-k]
=(1 + o(1))\E_{a_n(z^{\pm})}[f(a_n(z^{\pm}),k,W)||\tau^{bm}>1].
\]
In all, we have shown that it suffices to prove
\begin{equation}\label{eq:help6meander}
\E_{a_n(z^{\pm})}[f(a_n(z^{\pm}),k,W)|\tau^{bm}>1] \ra \E_{(0,0)}[f(W)|\tau^{bm}>1],\quad n\ra \infty
\end{equation}
uniformly for $k\leq n^{1-\epsilon}, z\in K_{n,\epsilon}$.
Notice that \eqref{eq:convintbm}, together with the mild condition \eqref{eq:mild}, implies immediately for every $\beta>0$
\begin{equation}\label{eq:thm12short}
\lim_{n\ra \infty}\sup_{z\in\R_+\times\R_+,\alpha(z)\leq n^{-\beta}}\left|\E_{z}[f(W)|\tau^{bm}>1]-\E_{(0,0)}[f(W)|\tau^{bm}>1]\right| = 0.
\end{equation}

Note also that for $z\in L_{n,\epsilon}$ we have $a_n(z) = O(n^{-\beta})$ uniformly for some $\beta>0$ small, if we choose $\theta_n = n^{-p}$ for some $p>0$ small enough.

We now show that
\begin{equation}\label{eq:help123}
|\E_{a_n(z^{\pm})}[f(a_n(z^{\pm}),k, W)|\tau^{bm}>1]- \E_{a_n(z^{\pm})}[f(W)|\tau^{bm}>1] |\ra 0,
\end{equation}
uniformly in $k\leq n^{1-\epsilon}, z\in L_{n,\epsilon}$ to complete the proof of \eqref{eq:help6meander} by the use of triangle inequality.

The difference of arguments of $f(a_n(z^{\pm}),k, W)$ and $f(W)$ is up to a constant bounded above by
\[
\max_{t,s\in[0,1],|t-s|\leq n^{-\epsilon}}|W_t-W_s| \textrm{ uniformly for all } k\leq n^{1-\epsilon}.
\]
Recalling that $f$ is uniformly continuous, we can write for each $\delta'>0$
\begin{align*}
&|\E_{a_n(z^{\pm})}[f(a_n(z^{\pm}),k, W)|\tau^{bm}>1]- \E_{a_n(z^{\pm})}[f(W)|\tau^{bm}>1] |\\&\leq C\delta' + 2\p_{a_n(z^{\pm})}(\max_{|t-s|\leq n^{-\epsilon}}|W_t-W_s|\geq \delta|\tau^{bm}>1),
\end{align*}
for some $\delta>0$ small enough.
We now show that
\begin{equation}\label{eq:help7meander}
\limsup_{n\ra\infty}\p_{a_n(z^{\pm})}\left(\max_{|t-s|\leq n^{-\epsilon}}|W_t-W_s|\geq \delta|\tau^{bm}>1\right)=0.
\end{equation}
Since the sets 
\[
\Gamma_n=\{\max_{|t-s|\leq n^{-\epsilon}}|W_t-W_s|\geq \delta\},
\]
are closed subsets of the metric space $\left(C[0,1],||\cdot||_{\infty}\right)$, and moreover these sets are monotonously decreasing in $n$ we can always write for fixed $m\in\N$
\[
\limsup_{n\ra\infty}\p_{a_n(z^{\pm})}\left(\max_{|t-s|\leq n^{-\epsilon}}|W_t-W_s|\geq \delta|\tau^{bm}>1\right)\leq \p_{(0,0)}(\Gamma_m|\tau^{bm}>1),
\]
uniformly for $z\in L_{n,\epsilon}$.
The measure $\p_{(0,0)}(\cdot\textrm{ }|\tau^{bm}>1)$ is gained by convergence of densities from measures which are absolutely continuous w.r.t. to the distribution of $W$ on $C[0,\infty)$. This implies that $\p_{(0,0)}(\cdot\textrm{ }|\tau^{bm}>1)$ is also absolutely continuous w.r.t. to the distribution of $W$ on $C[0,\infty)$. As a consequence, $\Gamma_m\searrow\emptyset,m\ra\infty$ $\quad \p_{(0,0)}(\cdot\textrm{ }|\tau^{bm}>1)$-a.s. as well.
This shows \eqref{eq:help7meander}, which in turn implies \eqref{eq:help123}. \eqref{eq:thm12short} and \eqref{eq:help123} imply \eqref{eq:help6meander}, which in turn shows \eqref{eq:help2meander}. This finishes the proof of Theorem \ref{thm:meander}.
\\\indent Proof of Corollary \ref{thm:meandert} is straightforward from a time rescaling. 

\section{Proof of invariance principle for bridges}
We now assume that the random walk moves on $\Z$ and is aperiodic, as described in subsection \ref{subsubsec:assumptionsint}. This section proves Theorem \ref{thm:bridgeinv}.

As a first step, we note the following easy property.
\begin{lemma}[Time reversal]\label{thm:timereversal}
 For all $z = (z_1,z_2),y=(y_1,y_2)\in\Z_+\times\Z, n\in \N$ define $\tilde z =(z_1,-z_2),\tilde y=(y_1,-y_2)$.  We have the following relation
\[
\p_z(Z(n)=y,\tau>n)= \p_{\tilde{y}}(\tilde{Z}(n) = \tilde{z},\tilde{\tau}>n),
\]
where $\tilde{Z}$ is the two-dimensional Markov chain started at $\tilde y$, that satisfies the law of motion
\begin{equation}
    \label{eq:dynamicintmod}
    \tilde Z(l+1) = 
    \begin{pmatrix}
    1 & 1\\
    0 & 1
    \end{pmatrix} \tilde Z(l) + \tilde X_{l+1}\begin{pmatrix}
    0\\
    1
    \end{pmatrix},\quad l\ge 0, \{\tilde X_s, s\ge 1\}\textrm{ are i.i.d},
\end{equation}
with $\tilde X_1$ distributed as $X_1$. 

$\tilde Z$ behaves asymptotically like the Markov chain in \eqref{eq:dynamicint} started at $\tilde y$.
\end{lemma}
\begin{proof}
We have $(T_0,S_0) = (z_1,z_2)$, the recursion 
\begin{equation}\label{eq:recursion1}
T_l = T_{l-1}+S_l, \quad S_l = S_{l-1}+X_l, \quad 1\leq l\leq n.
\end{equation}
and $(T_n,S_n) = (y_1,y_2)$. Define $(\tilde{T}_0,\tilde{S}_0) = (y_1,-y_2)$
and the recursion 
\begin{equation}\label{eq:recursion2}
\tilde{T}_{l} = \tilde{T}_{l-1} +\tilde{S}_{l-1},\quad  \tilde{S}_{l} = \tilde{S}_{l-1}+X_{n-l+1},\quad 1\leq l\leq n. 
\end{equation}
Then we have $\tilde{S}_{l} = -S_{n-l}$ and $\tilde{T}_l = T_{n-l}$. Therefore, this path transformation is such that the condition of positiveness of $\tilde{T}$ process started at $(y_1,-y_2)$ and ending in $(z_1,-z_2)$
in $n$-steps is preserved throughout, whenever it is preserved for the original process $T$ started at $(z_1,z_2)$
and ending at $(y_1,y_2)$. 


Even more is true: there is a one-to-one and measure-preserving map between paths of $Z$ started at $x$ and ending at $y$ at time $n$, and paths of $\tilde Z$ started at $(y_1,-y_2)$ and ending at $(z_1,-z_2)$
at time $n$. In particular, one can gain results about the asymptotic behavior of $\tilde Z$, conditioned on not leaving $\R_+\times\R$ until time $n$, from the ones about $Z$ under the same condition.  
\end{proof}

\begin{remark}
Lemma \ref{thm:timereversal} clarifies that in equation (13) in \citet{dw12} the function $V'$ is related to the harmonic function for the Markov chain as in \eqref{eq:dynamicint}, where the innovations of the underlying random walk are distributed as $-X_1$. Namely, if $\tilde V$ is the said harmonic function, then $V'(y) = \tilde V((y_1,-y_2))$ for $(y_1,y_2)\in\R_+\times\R$. 
\end{remark}

We show: given $x=(x_1,x_2), y=(y_1,y_2)\in\Z_+\times\Z$ and $0\leq f\leq 1$ bounded and uniformly continuous and measurable w.r.t. $(D[0,t],||\cdot||_{\infty})$ for some $t\in(0,1)$ we have 
\[
\E_{a_n(x_1,x_2)}[f(X_n)|\tau>n, X_n(1)=a_n(y_1,y_2)]\ra \E_{(0,0)}[f(W)f_t(W_t)|\tau^{bm}>1],\quad n\ra\infty,
\]
with some $f_t$ bounded and continuous. In particular, this also establishes a density for the limit process of the scaled bridges up until time $t$.

As a first step note that for $\mathcal{C}$ measurable w.r.t. Borel $\sigma$- algebra on $D[0,t]$ the following holds.
\begin{align*}
\p_{(x_1,x_2)}[\mathcal{C}|\tau>n, T_n=y_1, S_n = y_2] &= \frac{\E_x[\textbf{1}_{\mathcal{C}},\tau>nt,\p_{Z(nt)}(Z(n(1-t))=y,\tau>n(1-t))]}{\p_{x}(\tau>n, T_n=y_1, S_n = y_2)}\\
&=\E_x[\textbf{1}_{\mathcal{C}} f_t^n(X_n(t),a_n(x),a_n(y)|\tau>nt],
\end{align*}
with  the definition
\begin{align*}
\R_+\times\R\ni z\mapsto f_t^n(z,x,y) &= \frac{\p_z(X_n(1-t)=a_n(y),\tau>n(1-t))\p_x(\tau>nt)}{\p_{x}(\tau>n, X_n(1) = a_n(y))}\\
&= \frac{\p_{b_n(z)}(Z(n(1-t))=y,\tau>n(1-t))\p_{x}(\tau>nt)}{\p_{x}(\tau>n, Z(n) = y)},
\end{align*}
where $b_n(\tilde z) = a_n^{-1}(\tilde z)$ for $\tilde z\in \Z^2$.
One can show the following Lemma with the help of relations in \citet{dw12}.
\begin{lemma}\label{thm:density}
For $x,y\in\Z_+\times\Z$ we have 
\[
\lim_{n\ra\infty} \sup_{z\in \Z_+\times\Z,}|f_t^n(z,x,y)-f_t(z)| = 0,
\]
where $f_t (z) = \varkappa^2t^{-\frac{1}{4}}(1-t)^{-2-\frac{1}{4}}h(z_1,-z_2)p_{1}(0,0;z_1,-z_2)$ is continuous and bounded.
\end{lemma}
\begin{proof}

We have
\begin{align*}
&f_t^n(z,x,y) = \frac{\p_{b_n(z)}(Z(n(1-t))=y,\tau>n(1-t))\p_x(\tau>nt)}{\p_x(\tau>n, T(n) = y_1,Z(n) = y_2)}\\
&= \frac{\p_{\tilde{y}}(\tilde{Z}(n(1-t))=b_n(\tilde{z}),\tau>n(1-t))\p_x(\tau>nt)}{\p_x(\tau>n, T(n) = y_1,S(n) = y_2)},
\end{align*}

where in the last step we used Lemma \ref{thm:timereversal}.\\\indent First, the conditional local limit theorem from \citet{dw12} (see relation $(12)$ there) yields uniformly for all $z\in\Z_+\times\Z$ (with $\tilde z = (z_1,-z_2)$)
\begin{align}\label{eq:help1bridge}
\nonumber
\p_{\tilde{y}}(\tilde{Z}(n(1-t))&=b_n(\tilde{z}),\tilde{\tau}>n(1-t))\\ &\sim \varkappa n^{-2-\frac{1}{4}}(1-t)^{-2-\frac{1}{4}}V'(y)h(\tilde z)p_{1}(0,0;\tilde z).
\end{align}
Note that we also used the scaling property of $p_t(x,y,u,v)$ here. 
\\\indent Second, we get by using (7) in \citet{dw12}
\begin{equation}\label{eq:help2bridge}
\p_x(\tau>nt)\sim\varkappa V(x)n^{-
\frac{1}{4}}t^{-\frac{1}{4}}.
\end{equation}
\indent Finally, we use (13) in \citet{dw12} 
\begin{equation}\label{eq:help3bridge}
\p_x(\tau>n, T(n) = y_1,S(n) = y_2)\sim n^{-2-\frac{1}{2}}V(x)V'(y).
\end{equation}
We get the result about $f_t$ by combining \eqref{eq:help1bridge}, \eqref{eq:help2bridge} and \eqref{eq:help3bridge}. 
\end{proof}

\vspace{3mm}

We continue with the proof of the convergence to the bridge. For a given functional $g$ on $(D[0,1],||\cdot||_{\infty})$, which is uniformly continuous and bounded, we can write with the triangle inequality 
\begin{align*}
&|\E_x[g(X)f_t^n(X_{n}(t),x,y)|\tau>nt]-\E_{(0,0)}[g(W)f_t(W_t)|\tau^{bm}>t]|\\
&\leq \E_x[g(X)|f_t^n(X_{n}(t),x,y)-f_t(X_{n}(t))||\tau>nt] \\
&+\left|\E_x[g(X)f_t(X_{n}(t))|\tau>nt]-\E_{(0,0)}[g(W)f_t(W_t)|\tau^{bm}>t]\right|\\
&=: A_n + B_n.
\end{align*}
$A_n = o(1)$ follows from Lemma \ref{thm:density} and the boundedness of $g$. $B_n = o(1)$ can be proven using Skorohod's representation theorem. For more details, take some suitable probability space and define there $X_{n}(\cdot)$ with distribution $\p(\cdot\textrm{ }|\tau>nt)$ and $W_{\cdot}$ with distribution $\p(\cdot\textrm{ }|\tau^{bm}>t)$ so that $X_{n}(\cdot)\ra W_{\cdot},\ n\ra \infty$ a.s. This is possible due to Corollary \ref{thm:meandert}. It then follows due to continuity and boundedness of $f_t$ and $g$ with the new measure, whose expectation we denote by $E$, that
\begin{align*}
&|\E_x[g(X)f_t(X_{nt})|\tau>nt]-\E_{(0,0)}[g(W)f_t(W_t)|\tau^{bm}>t]|\\& = |\E[g(X)f_t(X_{nt})] - \E[g(W)f_t(W_t)]|\\
& \leq C\E[|f_t(X_{nt})-f_t(W_t)|] + \E[f_t(W_t)|g(W)-g(X_{nt})|].
\end{align*}
Dominated convergence, together with boundedness and continuity of $f,g$ gives the result. 

In all, we have proven the following.
\begin{lemma}\label{thm:convpartial}
For each $t\in (0,1)$ and $x,y\in \Z_+\times\Z$ it holds on $D[0,t]$ that
\[
(X_n|X_n(1)=a_n(y),\tau>n)_{a_n(x)} \textrm{ converges weakly as }n\ra\infty.
\]
The limit process $Y|_{D[0,t]}$ is absolutely continuous w.r.t. Kolmogorov meander measure $\p_0(W\in\cdot\textrm{ }|\tau^{bm}>t)$ and has continuous and bounded density $f_t$. 
\end{lemma}
In particular, this implies tightness of the measures $(X_n|X_n(1)=a_n(y),\tau>n)_{a_n(x)}$ in $D[0,t]$. Since we know the density of $\p_0(W_t\in\cdot|\tau^{bm}>t)$ (see Proposition \ref{thm:convergencekolmogorovmeander} in the appendix), we get a Lebesgue density for $Y_t$.
\begin{corollary}\label{thm:densitylimit}
$Y_t$ has Lebesgue density
\begin{align*}
g_t(z) &= f_t(z)p_t^+(z_1,z_2) \\&= \varkappa t^{-\frac{1}{4}}(1-t)^{-2-\frac{1}{4}}h(z_1,-z_2)p_1(0,0;z_1,-z_2)\pr_{(z_1,z_2)}(\tau^{bm}>1-t)\bar h(t,u,-v). 
\end{align*}
\end{corollary}

It is obvious from the construction and the weak convergence that the laws of $(Y_s)_{s\in(0,1)}$ are consistent, in the sense that for $s<t$ we have 
\[
Y_t|_{C[0,s]} \stackrel{d}{=} Y_s.
\]
Moreover, if we take the process $\tilde Z$ as defined in \eqref{eq:dynamicintmod} from Lemma \ref{thm:timereversal}, and denote by $\left(\tilde{X}^{(n)}\right)_{n\in\Z_+}$, the process defined through $\tilde X^{(n)}(t) = a_n(\tilde Z([nt])), t\in [0,1]$ then the following result follows immediately. 
\begin{lemma}
For each $t\in (0,1)$ and $x,y\in\R_+\times\R$ we have on $\left(D[0,t],||\cdot||_{\infty}\right)$
\[
(\tilde{X}_n|\tilde{X}_n(1)=a_n(y),\tau>n)_{a_n(x)}\textrm{ converges weakly as }n\ra\infty.
\]
\end{lemma}

From this we get tightness of $(\tilde{X}_n|\tilde{X}_n(1)=a_n(y),\tau>n)_{a_n(x)}$ in $D[0,t]$, which amounts to tightness of the processes $(X^{(n)}|X^{(n)}(1)=a_n(y),\tau>n)_{a_n(x)}$ in $D[1-t,1]$. Using the characterizations from Theorems 15.3 and 15.5 from \citet{billingsley} we get tightness on the whole $D[0,1]$. In all, we get weak convergence to a process in $D[0,1]$, whose density at any point is characterized in Corollary \ref{thm:densitylimit}. From Theorem 15.5 in \citet{billingsley} we get automatically that the weak limit has to be continuous, i.e. an element of $C[0,1]$.

\section{Proof of convergence of $h$-transforms}

\subsection{Proof of \eqref{eq:htransformKol}.}

Let $f$ be a uniformly continuous, bounded function on $C[0,1]$ with values in $[0,1]$. Fix some $S>0$. It holds 
\begin{equation}
\E_{(x,y)}^{(h)}[f(W)] = \E_{(x,y)}^{(h)}[f(W), |W(1)|\leq S]+ \E_{(x,y)}^{(h)}[f(W), |W(1)|> S].
\end{equation}
We have the following estimate with the help of \eqref{eq:needed} for all $(x,y)\in\R^2_+$ such that $\alpha(x,y)\le 1$. 
\begin{align*}
\E_{(x,y)}^{(h)}[f(W), |W(1)|> S]& = \frac{1}{h(x,y)}\E_{(x,y)}[f(W)h(W(1)),\tau^{bm}>1, |W(1)|> S]\\
&\leq C\int_{(u,v)\in\R_+\times\R:|(u,v)|>S}|\psi(u,v)h(u,v)|d(u,v),
\end{align*}
uniformly for $|(x,y)|\leq \frac{1}{2}, x,y>0$ and $S>0$.

In particular, it follows from the integrability of $\psi\cdot h$ that
\begin{equation}\label{eq:convhtransformsKDuniformzero}
\sup_{|(x,y)|\leq \frac{1}{2}, x,y>0}\E_{(x,y)}^{(h)}[f(W), |W(1)|> S] = o(1),\quad S\ra \infty.
\end{equation}
We have 
\begin{align*}
\E_{(x,y)}^{(h)}[f(W), |W(1)|\leq S]& = \frac{1}{h(x,y)}\E_{(x,y)}[f(W)h(W(1)),\tau^{bm}_{(x,y)}>1, |W(1)|\leq S]\\&
=\frac{\p_{(x,y)}(\tau^{bm}>1)}{h(x,y)}\E_{(x,y)}[f(W)h( W(1)),|W(1)|\leq S|\tau^{bm}>1].
\end{align*}
We recall \eqref{eq:helpk1}. Moreover, the function $$f(W)h( W(1))\textbf{1}_{\{|W(1)|\leq S\}}$$ is bounded and the set of its discontinuities is a null set with respect to the meander measure $\p_{(0,0)}(\cdot|\tau^{bm}>1)$ constructed by the weak convergence of $\p_{(x,y)}(\cdot|\tau^{bm}>1)$ as $x,y\searrow 0$. In particular, 
\begin{equation*}
\E_{(x,y)}[f(W)h( W(1)),|W(1)|\leq S|\tau^{bm}>1]\ra \E_{(0,0)}[f(W)h(W(1)),|W(1)|\leq S|\tau^{bm}>1],
\end{equation*}
because of the convergence of the meander of Kolmogorov diffusion as the start point converges to $(0,0)$. Monotone convergence delivers
\begin{equation}
\E_{(0,0)}[f(W)h(W(1)),|W(1)|\leq S|\tau^{bm}_{(0,0)}>1]\ra \E_{(0,0)}[f(W)h(W(1))|\tau^{bm}_{(0,0)}>1], \quad S\ra \infty.
\end{equation}

Combining this with \eqref{eq:convhtransformsKDuniformzero} shows 
\[
\E_{(x,y)}^{(h)}[f(W)]\ra \varkappa\E_{(0,0)}[f(W)h(W(1))|\tau^{bm}_{(0,0)}>1].
\]
This finishes the proof.

\subsection{Proof of Theorem \ref{thm:htransformintegrated}.}
We first note some sharp estimates between the harmonic function $V$ as constructed in \citet{dw12}, and the harmonic function $h$ for the Kolmogorov diffusion conditioned to stay positive. 

\begin{lemma}
\label{thm:hVestimate} Assume that $\E[|X|^{\frac{5}{2}+\epsilon_0}]<\infty$ for some $\epsilon>0$. It holds
\begin{equation}
    \label{eq:Vh-est}
    \left|V(z)-h(z)\right|\le C(1+\min\{1,\alpha(z)^{-\frac{3}{2}-\delta}\}),\quad z\in\R_+\times\R.
\end{equation}


\end{lemma}

\begin{proof}
This is an implication of the arguments in the proof of Lemma 11 in \citet{dw12} and Lemma 7 in \citet{dw12}. For some details, the proof of Lemma 11 there establishes that 

\[
\E_z[h(Z(n)),\tau>n]\le h(z)+\sum_{l=0}^{n-1}\E[|f(Z(l))|],
\]
for the function $f(z) = \E_z[h(Z(1))]-h(z), z\in \R^2$ (see (9), (30) and (31) in \citet{dw12}). Lemma 11 goes on to establish that 
\begin{equation}
    \label{eq:dw12help}
    \sum_{l=0}^{n-1}\E[|f(Z(l))|]\le C(1 + n^{\frac{1}{4}-\frac{\delta}{2}}),
\end{equation}
which delivers 
\begin{equation}
    \label{eq:dw12help2}
    |\E_z[h(Z(n)),\tau>n]-h(z)|\le f(z)+C(1 + n^{\frac{1}{4}-\frac{\delta}{2}}). 
\end{equation}
We note here that the equation at the top of page 181 in \citet{dw12} is actually equivalent to \eqref{eq:dw12help} whenever $\delta>0$ is small enough. The proof of Lemma 11 in \citet{dw12} establishes in fact the more general inequality \eqref{eq:dw12help}.   
\eqref{eq:dw12help2} together with Lemma 7 in \citet{dw12} delivers the statement.
\end{proof}

Lemma \ref{thm:hVestimate} has the following important Corollary, which we use repeatedly throughout the proof. 

\begin{corollary}
Assume that $\E[|X|^{\frac{5}{2}+\epsilon_0}]<\infty$. Then the following inequalities hold true for all $\Delta>0$ small enough.


\begin{equation}
    \label{eq:imp1Vh-est}
    |V(z)-h(z)|\le C\alpha(z)^{\frac{1}{2}-\Delta},\quad \alpha(z)\ge 1, z\in \R_+\times\R. 
\end{equation}
In particular, it follows 
\begin{equation}
    \label{eq:imp2Vh-est}
    V(z)\le C\alpha(z)^{\frac{1}{2}},\quad \alpha(z)\ge 1,z\in\R_+\times\R.
\end{equation}

\end{corollary}
The Corollary follows directly from Lemma \ref{thm:hVestimate} using Lemma 6 in \citet{dw12}. 


To prove Theorem \ref{thm:htransformintegrated}, we want to show that given a $f\in D[0,1]$ bounded and uniformly continuous and $(x,y)\in\R_+\times\R_+$ fixed, we have 
\[
\E_{(x,y)}^{(V)}[f(X^{(n)})]\ra \E_{(0,0)}^{(h)}[f(W)],\quad n\ra \infty.
\]

Assume w.l.o.g. that values of $f$ are in $[0,1]$. Then, 
\[
\E_{(x,y)}^{(V)}[f(X^{(n)}),\gamma_n\geq n^{1-\epsilon}]\leq \frac{1}{V(x,y)}\E_{(x,y)}[V(Z(n)),\tau>n^{1-\epsilon},\gamma_n\geq n^{1-\epsilon}].
\]
Using Markov property and the harmonicity of $V$, we get 
\begin{align*}
&\E_{(x,y)}[V(Z(n)),\tau>n^{1-\epsilon},\gamma_n\geq n^{1-\epsilon} ]\\& = \E_{(x,y)}\left[\E_{(x,y)}[V(Z(n),\tau>n|\mathcal{F}_{n^{1-\epsilon}}],\tau>n^{1-\epsilon},\gamma_n\geq n^{1-\epsilon}\right]\\&= \E_{(x,y)}[V(Z(n^{1-\epsilon})),\gamma_n\geq n^{1-\epsilon},\tau\geq n^{1-\epsilon}].
\end{align*}
Recalling \eqref{eq:imp2Vh-est}, we use H\"older inequality to obtain 
\begin{align*}
    \E_{(x,y)}&[V(Z(n^{1-\epsilon})),\gamma_n\geq n^{1-\epsilon},\tau\geq n^{1-\epsilon}]\\&
    \le C\E_{(x,y)}[\alpha(Z(n^{1-\epsilon}))^{}]^{\frac{1}{2}}\left(\pr_{(x,y)}(\gamma_n\geq n^{1-\epsilon},\tau\geq n^{1-\epsilon})\right)^{\frac{1}{2}}.
\end{align*}
Note that $\alpha(Z(n^{1-\epsilon}))\le 2\max\{n^{(1-\epsilon)\frac{1}{3}}M(n^{1-\epsilon})^{\frac{1}{3}},M(n^{1-\epsilon})\}$. 

Together with Doob's inequality for martingales and Lemma 12 in \citet{dw12} this delivers 
\[
\E_{(x,y)}^{(V)}[f(X^{(n)}),\gamma_n\geq n^{1-\epsilon}] = o(1),\quad n\ra \infty.
\]


We look at
\begin{align*}
\E_{(x,y)}^{(V)}[f(X^{(n)}),\gamma_n\leq n^{1-\epsilon}]& = \E_{(x,y)}^{(V)}[f(X^{(n)}),\gamma_n\leq n^{1-\epsilon},m_{\gamma_n}>\theta_n\sqrt{n}]\\&+ \E_{(x,y)}^{(V)}[f(X^{(n)}),\gamma_n\leq n^{1-\epsilon},m_{\gamma_n}\leq\theta_n\sqrt{n}]\\
=A^{(n)}+B^{(n)}.
\end{align*}
Using Lemma \ref{thm:hVestimate} here together with Lemma 6 in \cite{dw12} for paths such that $\alpha(Z(n^{1-\epsilon}))\le 1$,  and  \eqref{eq:imp2Vh-est} for paths such that $\alpha(Z(n^{1-\epsilon}))> 1$, $A^{(n)}$ can be estimated as follows for $\theta_n\ra 0$ slow enough 
\begin{align*}
&\E_{(x,y)}^{(V)}[f(X^{(n)}),\gamma_n\leq n^{1-\epsilon},m_{\gamma_n}>\theta_n\sqrt{n}]
\\& \le \frac{1}{V(x,y)}\E_{(x,y)}[V(Z(n^{1-\epsilon})),\alpha(Z(n^{1-\epsilon}))\le 1,m_{n^{1-\epsilon}}>\theta_n\sqrt{n}]\\&+\frac{1}{V(x,y)}\E_{(x,y)}[V(Z(n^{1-\epsilon})),\alpha(Z(n^{1-\epsilon}))> 1,m_{n^{1-\epsilon}}>\theta_n\sqrt{n}] \\&\leq \frac{C_{x,y}}{V(x,y)}(\p(m_{n^{1-\epsilon}}>\theta_n\sqrt{n})+\E_{(x,y)}[m_{n^{1-\epsilon}}^{\frac{1}{2}},m_{n^{1-\epsilon}}>\theta_n\sqrt{n}]).
\end{align*}
Note that $\{m_{n^{1-\epsilon}}>\theta_n\sqrt{n}\}\subset \{M_{n^{1-\epsilon}}>\theta_n^3\sqrt{n}\}$. It holds
\[
\p(M_{n^{1-\epsilon}}>\theta_n^3\sqrt{n}) = o(1),
\]
due to Kolmogorov inequality. 
Moreover, if $\theta_n$ goes to zero slowly enough, we have 
\begin{equation*}
\max\{((n^{1-\epsilon})M_{n^{1-\epsilon}})^{\frac{1}{3}},M_{n^{1-\epsilon}}\}=M_{n^{1-\epsilon}}
\end{equation*}
on the set $\{M_{n^{1-\epsilon}}>\theta_n^3\sqrt{n}\}$ for all $n$ large enough. To show that $A^{(n)} = o(1)$ it thus suffices to show 
\begin{equation}\label{eq:eqhelp11}
\E[M_{n^{1-\epsilon}}^{\frac{1}{2}},M_{n^{1-\epsilon}}>\theta_n^3\sqrt{n}] = o(1).
\end{equation}
Simple combinations of the H\"older inequality and the Kolmogorov inequality can be used to show \eqref{eq:eqhelp11}. 
We now look at $B^{(n)}$. Define for $w\in D[0,1]$
\[
f(z,k,w) = f(z\textbf{1}_{\{t\leq \frac{k}{n}\}}+ w(t)\textbf{1}_{\{t> \frac{k}{n}\}}).
\]
It follows from uniform continuity of $f$, just as for the proof of the invariance principle for the meander, that 
\begin{align*}
&\E_{(x,y)}^{(V)}[f(X^{(n)}),\gamma_n\leq n^{1-\epsilon},m_{\gamma_n}\leq\theta_n\sqrt{n}]\\& 
= o(1) + \E_{(x,y)}^{(V)}[f(a_n(Z_{\gamma_n}),\gamma_n, X^{(n)}),\gamma_n\leq n^{1-\epsilon},m_{\gamma_n}\leq\theta_n\sqrt{n}].
\end{align*}
We can use Markov property to write 
\begin{align}
\label{eq:principal}
& \E_{(x,y)}^{(V)}[f(a_n(Z_{\gamma_n}),\gamma_n, X^{(n)}),\gamma_n\leq n^{1-\epsilon},m_{\gamma_n}\leq\theta_n\sqrt{n}]\\
\nonumber & = \sum_{k=0}^{n^{1-\epsilon}}\frac{1}{V(x,y)}\int_{m_{n,\epsilon}}\p_{(x,y)}((Z(k)\in dz,\tau>k,\gamma_n = k, L_k\leq \theta_n\sqrt{n})\\
\nonumber &\times\E_{z}[f(a_n(z),k,X^{(n)})V(Z(n-k)),\tau>n-k]
\end{align}
To continue, we first show that 
\begin{align}
\E_{z}[V(Z(n-k)),A_n^c,\tau>n-k] = o(h(z))\label{eq:first},\\
\E_{z}[h(W(n-k)),A_n^c,\tau^{bm}>n-k]= o(h(z))\label{eq:second}
\end{align}
uniformly for $k\leq n^{1-\epsilon},z\in K_{n,\epsilon}, \alpha(z)\leq \theta_n\sqrt{n}$.
For the first expression note
\begin{align*}
\E_z[V(Z(n-k))&,A_n^c,\tau>n-k]\\
&=\E_z[V(z+Z(n-k)),\alpha(z+Z(n-k))\leq n^{\frac{1}{2}+\eta},A_n^c,\tau>n-k]\\
&+\E_z[V(z+Z(n-k)),\alpha(z+Z(n-k))> n^{\frac{1}{2}+\eta},A_n^c,\tau>n-k]\\
&=A^{(n)}+B^{(n)}.
\end{align*}
We get with help of \eqref{eq:imp1Vh-est} and \eqref{eq:imp2Vh-est}
\begin{equation}\label{eq:Anhelp}
A^{(n)}\leq C(1+ n^{\frac{1}{4}+\frac{\eta}{2}})\p(A_n^c)\leq Cn^{\frac{1}{4}+\frac{\eta}{2} - r}.
\end{equation}
Here $r = \frac{\delta}{2}-2\gamma-\gamma\delta$ from Lemma \ref{thm:coupling}.
Furthermore, Lemma 6 in \cite{dw12} implies 
\begin{equation}
h(z)>cn^{\frac{1}{4}-\frac{\epsilon}{2}},\quad z\in K_{n,\epsilon}.
\end{equation}
It follows that $A^{(n)} = o(h(z))$, uniformly for $z\in K_{n,\epsilon}, \alpha(z)\le \theta_n\sqrt{n}$, whenever $\epsilon,\eta>0$ are chosen small enough.
By an analogous logic, we have 
\begin{align*}
B^{(n)}&\leq C\E_z[\alpha(Z(n-k))^{\frac{1}{2}},\alpha(Z(n-k))> n^{\frac{1}{2}+\eta}]\\&\leq C\E[M_n^{\frac{1}{2}}+\alpha(z)^{\frac{1}{2}}, M_n>n^{\frac{1}{2}+\frac{\eta}{2}}]\leq Ch(z)n^{-\frac{1}{4}+\frac{\epsilon}{2}}E[M_n^{\frac{1}{2}}+\alpha(z)^{\frac{1}{2}}, M_n>n^{\frac{1}{2}+\frac{\eta}{2}}].
\end{align*}
Note that because $\alpha(z)\le\theta_n\sqrt{n}$ with $\theta_n\ra 0$, we can focus on the term $M(n)$ in the above expression.

Integration by parts gives 
\begin{align}\label{eq:help1htransform}
\nonumber
&n^{-\frac{1}{4}+\frac{\epsilon}{2}}E[M_n^{\frac{1}{2}}, M_n>n^{\frac{1}{2}+\frac{\eta}{2}}]\\
&=n^{\frac{\epsilon}{2}+\frac{\eta}{4}}\p(M_n>n^{\frac{1}{2}+\frac{\eta}{2}}) + \frac{1}{4}n^{-\frac{1}{4}+\frac{\epsilon}{2}}\int_{n^{\frac{1}{2}+\frac{\eta}{2}}}^{\infty}x^{-\frac{1}{2}}\p(M_n>u)du
\end{align}
We recall Fuk-Nagaev-Borovkov inequalities (see Lemma 22 of \citet{rwcones} and \cite{borovkov} which shows that all Fuk-Nagaev inequalities remain
valid for partial maximas of sums of independent random variables):
\begin{equation}\label{eq:fuknagaev}
\p(|M_n|>u)\leq 2\left(\frac{en}{uv}\right)^{{\frac{u}{v}}}+n\p(|X(1)|>v).
\end{equation}
One can use this estimate with $y=xn^{-\frac{\eta}{4}}$ to show, that \eqref{eq:help1htransform} is asymptotically zero, uniformly for $z\in K_{n,\epsilon},\alpha(z)\le\theta_n\sqrt{n}$, if we choose $\epsilon>0$ small enough and then, given this $\epsilon$ we choose $\eta>0$ small enough.
In all, we have shown \eqref{eq:first}. \eqref{eq:second} follows along similar lines and uses estimates for Brownian motion instead of the Fuk-Nagaev-Borovkov inequalities.

We now use coupling. Note that 
\[
\{\tau_{z^-}^{bm}>n\}\cap A_n\subset \{\tau_z>n\}\cap A_n\subset \{\tau_{z^+}^{bm}>n\}\cap A_n.
\]
Consider the event 
\[
D_{n,k}: = \{x'-n^{\frac{3}{2}-\gamma}+ \int_0^{n-k}B(s)ds\ge 2 n^{\frac{3}{2}-\gamma}\}. 
\]
 Then on the event $D_{n,k}\cap A_n$, it holds

 \[
 T(n-k)\ge n^{\frac{3}{2}-\gamma}\quad \textrm{ ultimately in }n. 
 \]

Now it follows on $D_{n,k}\cap A_n$ that 
\begin{equation}
    \left|V(Z(n-k))-h(Z(n-k))\right|\le C \alpha(Z(n-k))^{1-\Delta},
\end{equation}
for all $\Delta>0$ sufficiently small. 

We use this and \eqref{eq:imp1Vh-est} to get for a $\Delta>0$ small,
\begin{align*}
\E_z&[f(a_n(z),k,X^{(n)})V(Z(n-k)),\tau>n-k, A_n]\\
&\geq \E_z[f(a_n(z),k,X^{(n)})h(Z(n-k)),\tau_{z^-}^{bm}>n-k,A_n\cap D_{n,k}]\\
&- C\E_z[\alpha(Z(n-k))^{\frac{1}{2}-\Delta},\tau_{z^-}^{bm}>n-k,A_n],
\end{align*}
uniformly for $k\leq n^{1-\epsilon},z\in K_{n,\epsilon}, \alpha(z)\leq \theta_n\sqrt{n}$.

Note that 
\begin{align*}
    \E_z&[\alpha(Z(n-k))^{\frac{1}{2}-\Delta},\tau_{z^-}^{bm}>n-k,A_n]\\&\le 
    C\E_z[|T(n-k)|^{\frac{1}{3}(\frac{1}{2}-\Delta)}+|S(n-k)|^{\frac{1}{2}-\Delta},\tau_{z^-}^{bm}>n-k,A_n]\\
    &\le C\E_{z^-}[\alpha(W(n-k))^{\frac{1}{2}-\Delta},\tau_{z^-}^{bm}>n-k,A_n] \\&+ C\left(n^{\frac{1}{3}(\frac{3}{2}-\gamma)(\frac{1}{2}-\Delta)}+ n^{(\frac{1}{2}-\gamma)(\frac{1}{2}-\Delta)}\right)\pr_{z^-}(\tau^{bm}>n-k)\\
    & =: A + B.
\end{align*}
It is clear that $B = o(u(z))$ as $n\ra\infty$, uniformly for all $k\leq n^{1-\epsilon},z\in K_{n,\epsilon}, \alpha(z)\leq \theta_n\sqrt{n}$. Here we have used, that 
\[
\p(\tau_{z^{\pm}}^{bm}>n) = \varkappa h(z)(1+o(1)),\text{ uniformly in }z\in K_{n,\epsilon},
\]
which is shown in the proof of Lemma 18 of \citet{dw12}.

Note that 
\begin{align*}
    A&\le C\E_{z}[\alpha(W(n-k))^{\frac{1}{2}-\Delta},\tau_{z}^{bm}>n-k,A_n] \\&= C(n-k)^{\frac{1}{4}(1-\Delta)}\E_{a_{n-k}(z)}[h(W(1))^{1-\Delta},\tau^{bm}_{a_{n-k}(z)}>1].
\end{align*}
From here, we use the scaling property of the Kolmogorov diffusion, to get 
for all $n$ large so that uniformly for all $k\leq n^{1-\epsilon},z\in K_{n,\epsilon}, \alpha(z)\leq \theta_n\sqrt{n}$ it holds $\alpha(a_{n-k}(z))\le 1$, that
\[
A\le C(n-k)^{\frac{1}{4}(1-\Delta)}h(a_{n-k}(z))\int_{\R_+\times\R} \psi(\sigma)h(\sigma)^{1-\Delta}d\sigma.
\]
We use the scaling property of the function $h$ to substitute $h(a_{n-k}(z)) = (n-k)^{\frac{1}{4}}h(z)$. This shows that $A = o(u(z))$ as $n\ra\infty$, uniformly for all $k\leq n^{1-\epsilon},z\in K_{n,\epsilon}, \alpha(z)\leq \theta_n\sqrt{n}$.

Next we prove, that 
\begin{align}\label{eq:eqhelphtransfconv}
\nonumber
\E_z[h(&Z(n-k)),\tau_{z^-}^{bm}>n-k,A_n]\\ &= \E_{z^-}[h(W(n-k)),\tau^{bm}>n-k,A_n] + o(h(z)),
\end{align}
uniformly for $k\leq n^{1-\epsilon},z\in K_{n,\epsilon}, \alpha(z)\leq \theta_n\sqrt{n}$.\\

We have for some $\theta_n\ra 0$ with help of Lemma 6 and Lemma 18 in \citet{dw12}
\begin{align}\label{eq:help2htransform}
\nonumber
\E_{z^-}&[h(W(n-k)),\alpha(W(n-k))\leq \theta_n \sqrt{n},\tau^{bm}>n-k,A_n] \\
&\leq C \theta_n^{\frac{1}{2}}n^{\frac{1}{4}}\p_{z^-}(\tau^{bm}>n-k) = o(h(z)).
\end{align}
Here, we have used 
\begin{equation}\label{eq:help6htransform}
h(z^{\pm}) = (1+o(1))h(z),
\end{equation}
uniformly for $z\in K_{n,\epsilon}$, which is shown in the proof of Lemma 18 in \citet{dw12}.

Moreover, for all large enough $n$ and if $\theta_n$ goes to zero slowly enough, it holds
\[
\alpha(Z(n-k))\leq \frac{1}{2}\theta_n\sqrt{n},\quad \textrm{ on }A_n,\textrm{ whenever }\alpha(W(n-k))\leq \theta_n \sqrt{n}.
\]
Here $Z$ is started at $z$. We estimate 
\begin{align*}
    &\E_{z^-}[h(Z(n-k)),\alpha(W(n-k))\leq \theta_n \sqrt{n},\tau^{bm}>n-k,A_n]\\&\le
    \E_{z}[h(Z(n-k)),\alpha(Z(n-k))\leq \theta_n \sqrt{n},\tau>n-k,A_n].
\end{align*}
We use this last inequality and an estimate analogous to \eqref{eq:help2htransform} to get
\begin{equation}\label{eq:help3hitransform}
\E_{z^-}[h(Z(n-k)),\alpha(W(n-k))\leq \theta_n \sqrt{n},\tau^{bm}>n-k,A_n] = o(h(z)).
\end{equation}
For $\alpha(W(n-k))> \theta_n \sqrt{n}$ it holds again for $n$ large enough and $\theta_n$ going slowly enough to zero, that
\[
\alpha(Z(n-k))> \frac{1}{2}\theta_n\sqrt{n}.
\]

We now use Taylor formula together with 
 \eqref{eq:gradienth} to show that on $A_n\cap \{\alpha(W(n-k))> \theta_n \sqrt{n}\}$ for $n$ large enough and $\theta_n$ it holds 
\begin{equation}
    \label{eq:helpetehere}
    |h(Z(n-k))-h(W(n-k))|= o(n^{\frac{1}{4}}).
\end{equation}
Here $Z$ is started at $z$ and $W$ at $z^-$. 

To see \eqref{eq:helpetehere} note that in general we have 
\begin{equation}
    \label{eq:helpetehere2}
    |h(z+w)-h(z)| \le C\left( |w_1|\alpha(z+t_0w)^{-\frac{5}{2}}+|w_2|\alpha(z+t_0w)^{-\frac{1}{2}}\right),
\end{equation}
for some suitable $t_0\in [0,1]$. Namely, the proof of Lemma 6 in \citet{dw12} shows
\begin{equation}\label{eq:gradienth}
\left|\frac{\partial^{i+j}h}{\partial x^i\partial y^j}(x,y)\right|\leq C \alpha(x,y)^{\frac{1}{2}-3i-j}.
\end{equation}
Replace here $z$ by $W(n-k)$ and $w$ by $Z(n-k)-W(n-k)$. It holds 

\[
\alpha(w)\le C n^{\frac{1}{2}-\frac{1}{3}\gamma},
\]
and
\[
\alpha(z+t_0w)\ge C\alpha(z)-\alpha(w)\ge C\theta_n \sqrt{n},
\]
for all $n$ large enough, whenever $\theta_n$ goes to zero slowly enough. Here we have used the inequality $\alpha(z_1+z_2)\le C\left(\alpha(z_1)+\alpha(z_2)\right)$ for $z_1,z_2\in \R^2$.

We combine this with \eqref{eq:helpetehere2} to arrive at 
\[
\left|h(Z(n-k))-h(W(n-k))\right|\le C n^{\frac{1}{4}-\gamma}\left(\theta_n^{-\frac{5}{2}} + \theta_n^{-\frac{1}{2}}\right) = o(n^{\frac{1}{4}}),\quad n\ra\infty,
\]
whenever $\theta_n$ goes to zero slowly enough.

It follows that 
\begin{align*}
    \E_{z}&[|h(Z(n-k))-h(W(n-k)),\alpha(W(n-k))> \theta_n \sqrt{n},\tau_{z^-}^{bm}>n-k,A_n]\\&\nonumber = o(n^{\frac{1}{4}})\pr_{z^-}(\tau^{bm}>n-k) = o(h(z)),
\end{align*}
uniformly for $k\leq n^{1-\epsilon},z\in K_{n,\epsilon}, \alpha(z)\leq \theta_n\sqrt{n}$.
This finishes the proof of \eqref{eq:eqhelphtransfconv}.

We return to the main proof. So far we have shown
\begin{align*}
\E_z&[f(a_n(z),k,X^{(n)})V(Z(n-k)),\tau>n-k, A_n]\\
&\geq \E_z[f(a_n(z),k,X^{(n)})h(W(n-k)),\tau_{z^-}^{bm}>n-k,A_n\cap D_{n,k}] + o(h(z)),
\end{align*}
uniformly for $k\leq n^{1-\epsilon},z\in K_{n,\epsilon}, \alpha(z)\leq \theta_n\sqrt{n}$.

It holds 
\begin{align*}
    &\E_{z^-}[h(W(n-k)),\tau_{z^-}^{bm}>n-k,D_{n,k}^c] \\
    &\le (n-k)^{\frac{1}{4}}\pr_{z^-}(\tau^{bm}>n-k)\E_{a_{n-k}(z^{-})}[h(a_{n-k}(W(n-k))), \tilde D_{n,k}|\tau^{bm}>1],
\end{align*}
with $\tilde D_{n,k} = \{\frac{x'}{n^{\frac{3}{2}}}-n^{-\gamma}+ \int_0^{1}B(s)ds\le 4 n^{-\gamma}\}$. Due to convergence towards the meander, as well as the properties of the function $\psi$, the expectation converges to zero as $n\ra\infty$. Using Lemma 15 in \citet{dw12} together with \eqref{eq:help6htransform} it follows for all $n$ large enough that
\begin{align}
\label{eq:kickmefort}
\E_z&[f(a_n(z),k,X^{(n)})V(Z(n-k)),\tau>n-k, A_n]\\
\nonumber
&\geq \E_z[f(a_n(z),k,X^{(n)})h(W(n-k)),\tau_{z^-}^{bm}>n-k,A_n] + o(h(z)),
\end{align}
uniformly for $k\leq n^{1-\epsilon},z\in K_{n,\epsilon}, \alpha(z)\leq \theta_n\sqrt{n}$.


Analogous to above, we can prove
\begin{align}\label{eq:help5htransform}
\nonumber
\E_z&[f(a_n(z),k,X^{(n)})V(Z(n-k)),\tau>n-k,A_n]\\
&\leq \E_z[f(a_n(z),k,X^{(n)})h(W(n-k)),\tau_{z^+}^{bm}>n-k,A_n] + o(h(z)),
\end{align}
uniformly for $k\leq n^{1-\epsilon},z\in K_{n,\epsilon}, |z_2|\leq \theta_n\sqrt{n}$. 

In the proof of Theorem \ref{thm:meander} we have shown that 
\begin{equation}\label{eq:help30meander}
f(a_n(z),k,X^{(n)}) = f(a_n(z),k,W^{(n)}) +o(1),
\end{equation}
uniformly for $k\leq n^{1-\epsilon},z\in K_{n,\epsilon}, |z_2|\leq \theta_n\sqrt{n}$.
Here $W^{(n)} = a_n(W)$ is the rescaled Kolmogorov diffusion. \eqref{eq:help30meander} implies that 
\begin{align}
\label{eq:help31}
\E_z&[|f(a_n(z),k,X^{(n)})-f(a_n(z),k,W^{(n)})|h(W(n-k)),\tau_{z^\pm}^{bm}>n-k,A_n]\\&\nonumber = 
o\left(\E_{z^\pm}[h(W(n-k)),\tau^{bm}>n-k]\right) = o(h(z)),
\end{align}
uniformly for $k\leq n^{1-\epsilon},z\in K_{n,\epsilon}, |z_2|\leq \theta_n\sqrt{n}$. Here we have used the harmonicity of $h$ as well as \eqref{eq:help6htransform}.
From this we infer, after performing a scaling, that 
\begin{align*}
&h(z^-)\E^{(h)}_{a_n(z^-)}[f(a_n(z),k,W^{(n)})] + o(h(z))\\&\le
    \E_z[f(a_n(z),k,X^{(n)})V(Z(n-k)),\tau>n-k,A_n]\\&\le h(z^+)\E^{(h)}_{a_n(z^+)}[f(a_n(z),k,W^{(n)})] + o(h(z)),
\end{align*}
uniformly for $k\leq n^{1-\epsilon},z\in K_{n,\epsilon}, |z_2|\leq \theta_n\sqrt{n}$.

Given the convergence of the $h$-transforms for the Kolmogorov diffusion, we arrive at 
\begin{equation}
    \label{eq:karabush}
    \E_z[f(a_n(z),k,X^{(n)})V(Z(n-k)),\tau>n-k,A_n] = h(z)\E^{(h)}_0[f(W)] + o(h(z)),
\end{equation}
uniformly for $k\leq n^{1-\epsilon},z\in K_{n,\epsilon}, |z_2|\leq \theta_n\sqrt{n}$.
We return to \eqref{eq:principal}, which implies with \eqref{eq:karabush} that
\begin{align}
    \E_{(x,y)}^{(V)}&[f(a_n(Z(\gamma_n)),\gamma_n, X^{(n)}),\gamma_n\leq n^{1-\epsilon},m_{\gamma_n}\leq\theta_n\sqrt{n}] \\
    & = \frac{\E^{(h)}_0[f(W)] + o(1)}{V(x,y)}\E_{(x,y)}[h(Z(\gamma_n)),\gamma_n\le n^{1-\epsilon},\tau>\gamma_n,m_{\gamma_n}\le \theta_n\sqrt{n}]. 
\end{align}
A combination of Lemmas 20 and 21 from \citet{dw12} shows 
$$\lim_{n\ra\infty}\E_{(x,y)}[h(Z(\gamma_n)),\gamma_n\le n^{1-\epsilon},\tau>\gamma_n,m_{\gamma_n}\le \theta_n\sqrt{n}] = V(x,y).$$ This finishes the proof.

\section[Applications to the case with drift]{Applications to the case with drift}

The case $\E[X]\neq 0$ can be considered under two different settings. We develop the exponential case and just remark shortly on the case without exponential moments. \\\indent Consider a random variable which has finite exponential moments in a right neighborhood of zero: 
\[
M(t)=\E[e^{tX_1}]<\infty, \quad 0\leq t\leq t_0, \quad t_0>0.
\]
Denote by $\mu$ the law of $X_1$ and define the laws
\[
\mu_t(dz) = \frac{e^{tX_1}}{M(t)}d\mu(z).
\]
Let $X^{(t)}$ be a random variable with law $\mu_t$. Then we have for the laws $\mu_t^n$ of the respective vector $(X^{(t)}_1,\dots,X^{(t)}_n)$  under the change of exponential measure
\[
\mu_t^n(dz)= \frac{e^{tS^{(t)}(n)}}{M(t)^n}d\mu^n(dz),
\]
where $S^{(t)}(n)$ is the respective random walk.
Note that $f(t)=(\log{M(t)})' = \frac{\E[X_1e^{tX_1}]}{\E[e^{tX_1}]}$ for $t\in (0,t_0)$ is a bijective function in $t$ and increasing. Take some $c$ from its image, i.e. assume there exists $t_c$ with $f(t_c)=c$. Then it is easy to show that 
$\E[X^{(t_c)}] = c$. Note also that $\sigma^2=Var(X^{(t_c)}) = (\log{M(t)})''$ and that we can get an invariance principle for 
\[
X^{(n,t_c)}(s) = \left(\frac{T([ns])-\frac{ns(ns+1)}{2}c}{{\sigma^3 n}^{\frac{3}{2}}},\frac{S([ns])-nsc}{\sigma n^{\frac{1}{2}}}\right), \quad s\in[0,1]
\]
under the measure $\mu_t$. The condition  now is for the integrated random walk to stay strictly above a quadratic function. Define thus 
\[
\tau_c = \inf\{k\geq 0: T(k)\leq \frac{k(k+1)c}{2}\}.
\]

Then we have the following invariance principle for the meander.\footnote{Note that what we get is a whole family of invariance principles, indexed by $c$ in the image of $f(\cdot)$.}
\begin{theorem}
Let $X^{(n,t_c)}(s)$ for $s\in[0,1]$ as above under the measure $\mu_t$. Then these processes, started in $(x,y)\in\R_+\times\R_+$ and conditioned on $\tau_c>n$, converge to the meander of the Kolmogorov Diffusion, conditioned on not leaving $\R_+\times\R$ before time 1. 
\end{theorem}
Similarly, we get invariance principles for the case of bridges. 
\begin{theorem}
For $x,u\in \R_+, y,v\in \R$ we have ${(n,t_c)}(\cdot)|X^{(n,t_c)}=a_n(u,v),\tau>n)_{a_n(x,y)}$
converges weakly to the Kolmogorov excursion of length 1.
\end{theorem}
Finally, we remark the case with $\E[X_1]\neq 0$, where no exponential assumption is assumed.
\begin{remark}
Note that the above invariance principles also hold for random walks which instead fulfill a second moment condition of $\E[|X_1|^{2+\delta}]<\infty$ for some $\delta>0$ and have drift $\E[X_1]=c\neq 0$. The steps are the same as for the exponential case above and we don't need to make a change of measure. Instead we look at 
\[
X^{(n,c)}(s) = \left(\frac{T([ns])-\frac{ns(ns+1)}{2}c}{{Var[X_1]^{\frac{3}{2}} n}^{\frac{3}{2}}},\frac{S([ns])-nsc}{Var[X_1]^{\frac{1}{2}} n^{\frac{1}{2}}}\right), \quad s\in[0,1]
\]
in place of $X^{(n,t_c)}(s)$.
\end{remark}
\vspace{5mm}


\section{On Kolmogorov diffusion conditioned to stay positive}
\label{sec:kolmo}
\subsection{Introduction and results from the literature.} 
The Kolmogorov diffusion $W = (U,V)$ is a two dimensional real valued stochastic process with an integrated Brownian motion as a first coordinate and its integrand Brownian motion as a second coordinate. This process has a long history and the reader is referred to \citet{groeneboom} for more historical details and applications.
The transition density of this process for $x,y,u,v \in \mathbb{R}$ and $t > 0$ is given by
\begin{align}
\label{eq:pt}
\nonumber
p_t(x,y;u,v) &= \frac{\sqrt{3}}{\pi t^2} \exp \left( -\frac{6(u-x-ty)^2}{t^3} + \frac{6(v-y)(u-x-ty)}{t^2} - \frac{2(v-y)^2}{t} \right) \\
                     &= \frac{\sqrt{3}}{\pi t^2} \exp \left( -\frac{6(u-x)^2}{t^3} + \frac{6(u-x)(v+y)}{t^2} - \frac{2(v^2 + vy + y^2)}{t} \right).
\end{align}
The asymptotics of the exit time of this process from $\R_+\times\R$ is well-known (see \citet{groeneboom}):
\[
\p_{(x,y)}(\tau^{bm}>t)\sim \frac{h(x,y)}{t^{\frac{1}{4}}},\quad t\ra \infty.
\]
Here, $x>0,y\in \R$.
$h$ is the function $h:\R_+\times\R\longrightarrow \R$ which is harmonic for the Kolmogorov diffusion, in the sense that $\mathcal{A}h=0$ on $\R_+\times\R$, where $\mathcal{A} = y\frac{\partial}{\partial x}+\frac{1}{2}\frac{\partial^2}{\partial y^2}$ is the generator of $W$. In other words, $h(W(t))\textbf{1}_{\{\tau^{bm}>t\}}$ is a non-negative martingale. $h$ has the following integral representation. 
\begin{align}\label{eq:defh}
&h(x,y) := \int_{s=0}^\infty \int_{w=0}^\infty w^{3/2} q_s(x,y;0,-w) \textrm{d}s \textrm{d}w = \\
\nonumber
&\frac{2 \sqrt{3}}{\pi} \int_{s=0}^\infty \int_{w=0}^\infty w^{3/2} \exp (-6x^2 s^3 - 6xys^2 - 2(y^2 + w^2)s) \sinh(6xws^2 + 2yws) \textrm{d}s \textrm{d}w.
\end{align}
In the following, let $\bar{p}_t(x,y;u,v) = $ denote the density of the Kolmogorov diffusion started at $(x,y)$ at time $t$ and killed if the first coordinate becomes non-positive. 
We also introduce from \citet{groeneboom} the function $g: \mathbb{R} \rightarrow (0, \infty )$ defined by
\begin{equation}\label{eq:defg}
g(y) := h(1,y) = \int_{s = 0}^\infty \int_{w = 0}^\infty w^{3/2} q_s(1,y;0,-w) \textrm{d}s \textrm{d}w.
\end{equation}
In \citet{McK} the behavior of a Kolmogorov diffusion started in $(0,-|z|)$ that is stopped, when it hits zero is studied. \citet{L1} has done a more general analysis, where the process can start in any $(x,y)$ and is stopped when its first coordinate hits a given level.
\begin{lemma}[Joint distribution of $(\tau_a, V_{\tau_a})$-\citet{L1}]\label{thm:jointdistlemma}
The joint distribution of $(\tau_a, V_{\tau_a})$, where $\tau_a$ is the stopping time $\tau_a := \inf_{t>0} \{U_t = a\}$, is determined by the density $f_t$, which is recursively defined by
\begin{align*}
\textrm{1.} ~& f_s(0,-|z|;0,w) :=  \mathbb{P}_{(0, -|z|)}\big(\tau_0 \in \textrm{d}s, V_{\tau_0} \in \textrm{d}w\big)/\textrm{d}s\textrm{d}w = \\
\nonumber
=& \frac{3 |w|}{\pi \sqrt{2 \pi} s^2} e^{-(2/s)(z^2  -|zw| + w^2)} \int_0^{4 |zw| /s} \theta^{-1/2} e^{-(3 \theta / 2)} \textrm{d} \theta  \\
\nonumber
\textrm{and} \\
\textrm{2.} ~& f_t(x,y;a,z) := \mathbb{P}_{(x,y)}\big(\tau_a \in \textrm{d}t, V_{\tau_a} \in \textrm{d}z\big) /\textrm{d}t\textrm{d}z =  \\
\nonumber
=& ~ |z| \Bigl[ p_t(x,y,a,z) - 
 \int_0^t \int_0^\infty f_s(0,-|z|;0, w) p_{t-s}(x,y,a, -\varsigma w) \textrm{d}w \textrm{d}s \Bigr] 1_A(z), 
\end{align*}
where $A = [0, \infty)$ if $x<a$, $A = (-\infty, 0]$ if $x>a$ and $\varsigma$ is the sign of $(a-x)$. 
\end{lemma}
Using this, Lachal (see \citet[p.128]{L4}) described the transition density of $(U, V)$ killed when $U$ hits $a \in \mathbb{R}$. Note that 
$$\bar p_t(x,y;u,v)=\mathbb{P}_{(x,y)}\big((U_t,V_t) \in \textrm{d}u\textrm{d}v, \tau_0 > t\big).
$$
\begin{lemma}[Formula for $\overline{p}$]\label{thm:formulabarp}
For $t, x, u > 0$ and $y,v \in \mathbb{R}$ and where $A$ is again $(-\infty,0]$ if $x>a$ and $[0, \infty)$ if $x<a$ we have 
\begin{align*}
\overline{p}^a_t(x,y;u,v) &:= \mathbb{P}_{(x,y)}\big((U_t,V_t) \in \textrm{d}u\textrm{d}v, \tau_a > t\big) / \textrm{d}u\textrm{d}v =   \\
\nonumber
                          &= p_t(x,y;u,v) - \int_0^t \int_A  f_s(x,y;a,z) p_{t-s}(a,z;u,v) \textrm{d}z \textrm{d}s.
\end{align*}
\end{lemma}

\paragraph{\textbf{Representations in terms of confluent hypergeometric functions.}}
$h$ also has a representation in terms of confluent hypergeometric functions. Namely, in \citet{groeneboom} it is shown how the function $g$ can be expressed in terms of two such functions, specifically the Tricomi function $U$ and the function $V$. Next, we introduce these functions.

Tricomi's function $U$ is defined as the unique solution of the confluent hypergeometric equation with the convergence behavior 
\begin{equation*}\label{eq:hypergeometricU}
U(a,c,z) \sim z^{-a} \sum_{s=0}^{\infty} (-1)^s \frac{(a)_s (1+a-c)_s}{s!z^s} \textrm{~for~} z \rightarrow \infty ~ \textrm{in} ~|\textrm{ph}(z)| < \frac{3}{2}\pi,
\end{equation*}
where ph denotes the phase of a complex number and $( )_s$ is the so called Pochhammer's symbol defined as
\begin{equation*}
(a)_0 := 1 \textrm{~and~} (a)_s := a(a+1)(a+2) \dots (a+s-1) \textrm{~for~} s \in \mathbb{N}. 
\end{equation*}
The confluent hypergeometric function $V$ is defined as
\begin{equation*}\label{eq:hypergeometricV}
V(a,c,z) := e^z U(c-a,c,-z) ~ \textrm{for} ~ |\textrm{ph}(z)| < \frac{1}{2} \pi, \textrm{Re}(a) > 0.
\end{equation*}

See \citet{Olv}[p.256] for more on these definitions and further theory concerning confluent hypergeometric equations and functions. Now we can state the representation of $g$ (see \citet{groeneboom} for a proof).
\begin{lemma}
The function $g$ has the representation
\begin{align*}
g(y) &= ~(\frac{2}{9})^{1/6}y~U\left(\frac{1}{6},\frac{4}{3},\frac{2}{9}y^3\right)            &&\textrm{~for~} y>0, \\
g(y) &= -(\frac{2}{9})^{1/6}\frac{1}{6}y~V\left(\frac{1}{6},\frac{4}{3},\frac{2}{9}y^3\right) &&\textrm{~for~} y<0, \\
g(0) &= ~(\frac{2}{9})^{-1/6}\frac{\Gamma(1/3)}{\Gamma(1/6)}. &&
\end{align*}
\end{lemma}


\paragraph{\textbf{The auxiliary function $\bar h$}.} The statement of the results in this section requires the introduction of the auxiliary function $\bar{h}:\R_+\times\R_+\times\R\ra\R$
\[
\bar{h}(t,x,y) = \frac{4\sqrt{3}}{\sqrt{2\pi}}\int_0^t\int_0^{\infty}w^{\frac{3}{2}}s^{-\frac{1}{2}}p_s(0,w;0,0)\left(p_{t-s}(x,y;0,-w)-p_{t-s}(x,y;0,w)\right)dwds.
\]
With these definitions, we can state the main results of this section.

\begin{proposition}\label{thm:garbitlike}
There exists $R>0, g:\R_+\times\R\ra\R_+$ with the property $$\int_{\R_+}\int_{\R}g(u,v)h(u,v)dvdu<\infty,$$ so that
\[
\sup_{(x,y)\in\R_+\times\R_+,\alpha (x,y)\leq \frac{1}{2}}\bar{p}_1(x,y;u,v)\leq g(u,v)h(x,y),\quad\textrm{ for }(u,v)\in\R_+\times\R; u\geq R \textrm{ or }|v|\geq R.
\]
\end{proposition}

\begin{proposition}\label{thm:propbarp}
For $x,y \searrow 0$ we have 
\begin{equation*}
\overline{p}_t(x,y;u,v) \sim h(x,y) \overline{h}(t,u,-v),
\end{equation*}
uniformly in $(u,v)$ with $\alpha(u,v)\le R$ ($R>0$ arbitrary). 
\end{proposition}

Note that the proof of Proposition \ref{thm:propbarp} given here establishes formally also the result stated in Lemma 16 of \citet{dw12}.

Before continuing with the proofs of the Propositions \ref{thm:garbitlike} and \ref{thm:propbarp}, we give an application. We use Proposition \ref{thm:propbarp} to calculate the density of the meander $W_{(0,0)}(\cdot|\tau^{bm}>1)$. 

\begin{proposition}\label{thm:convergencekolmogorovmeander}
The density of the process $ W_{(0,0)}(\cdot\textrm{ }|\tau^{bm}>1)$ at time $t\in (0,1)$ is proportional to $\p_{(u,v)}(\tau^{bm}>1-t)\bar{h}(t,u,-v)$. 
\end{proposition}

\begin{proof}
We look at the convergence of densities, i.e. of 
\[
\frac{\bar{p}_t(x,y;u,v)\p_{(u,v)}(\tau^{bm}>1-t)}{\p_{(x,y)}(\tau^{bm}>1)}
\]
as $x,y\searrow 0$. We use Proposition \ref{thm:propbarp} $\overline{p}_t(x,y,u,v)\sim h(x,y)\overline{h}(t,u,-v)$. The scaling property of Brownian motion, implies $$\p_{(x,y)}(\tau^{bm}>1) = \p_{(\lambda^{\frac{3}{2}}x,\lambda^{\frac{1}{2}}y)}(\tau^{bm}>\lambda).$$ This implies in turn, that 
\begin{align*}
&\p_{(x,y)}(\tau^{bm}>1) = \p_{(1,yx^{-\frac{1}{3}})}(\tau^{bm}>x^{-\frac{2}{3}})\quad \textrm{ for }x^{\frac{1}{3}}\geq y,\\
&\p_{(x,y)}(\tau^{bm}>1) = \p_{(xy^{-3},1)}(\tau^{bm}>y^{-2})\quad \textrm{ for }x^{\frac{1}{3}}<y.
\end{align*}
Using uniform continuity of $h(1,t), h(t,1)$ as functions in $t$, Lemma 15 from \citet{dw12} with $\theta_t = \max\{x^{\frac{1}{6}},y^{\frac{1}{2}}\}, t = x^{-\frac{2}{3}}$ and the scaling properties of $h$ we get in all
\[
\p_{(x,y)}(\tau^{bm}>1) \sim \varkappa h(x,y), \textrm{ as }x,y\searrow 0
\]
The limit density is $C\p_{(u,v)}(\tau^{bm}>1)\bar{h}(t,u,-v)$, where $C$ is a suitable normalization constant.
\end{proof}

We calculate here also the density of the meander of length $t\neq 1$. For this we note that if the meander is started at some point $(x,y)\in \R_+\times\R_+$ then the density at some time $s\leq t$ is 
\[
p_t^+(s;x,y,;u,v)=\frac{\bar{p}_s(x,y;u,v)\p_{(u,v)}(\tau^{bm}>t-s)}{\p_{(x,y)}(\tau^{bm}>t)}.
\]
Here we can write
\[
\p_{(x,y)}(\tau^{bm}>t) = \p_{(xt^{-\frac{3}{2}},yt^{-\frac{1}{2}})}(\tau^{bm}>1)\sim h(xt^{-\frac{3}{2}},yt^{-\frac{1}{2}}) = t^{-\frac{1}{4}}h(x,y), \textrm{ as }x,y\searrow 0
\]
Using then again $\overline{p}_s(x,y;u,v)\sim h(x,y)\overline{h}(s,u,-v)$ from above we get
\[
p_t^+(s;x,y,;u,v)\sim t^{\frac{1}{4}}\bar{h}(s,u,-v)\p_{(u,v)}(\tau^{bm}>t-s)\textrm{ as }x,y\searrow 0.
\]
In all, we have proven the following.
\begin{corollary}\label{thm:densitymeander-t}
The density of meander of length $t\in(0,\infty)$ with start point in $0$ is 
\[
p_t^+(s;u,v) = Ct^{\frac{1}{4}}\bar{h}(s,u,-v)\p_{(u,v)}(\tau^{bm}>t-s),\quad s\leq t, (u,v)\in\R_+\times\R.
\]
\end{corollary}

\subsubsection{Auxiliary Lemmas.}

As a next step, we analyze the convergence behavior of the function $h$.
\begin{lemma}[Asymptotic behavior of $h$]\label{thm:convbehaviorh}
For $x \downarrow 0$ and $y \rightarrow 0$ we have
\begin{enumerate}
\item
$h(x,y) \sim x^{1/6}g(c)$ as $x \downarrow 0$ and $y \rightarrow 0$ such that $x^{-1/3}y \rightarrow c$ for $c \in \mathbb{R}$.
\item
$h(x,y) \sim y^{1/2}$ as $x \downarrow 0$ and $y \rightarrow 0$ such that $x^{-1/3}y \rightarrow +\infty$.
\item
$h(x,y) \sim \frac{3}{4}x(-y)^{-5/2}e^{\frac{2}{9}\frac{y^3}{x}}$ as $x \downarrow 0$ and $y \rightarrow 0$ such that $x^{-1/3}y \rightarrow -\infty$.
\end{enumerate}  
\end{lemma}

\begin{proof}
A. This is a simple consequence of the representation of $h$ and $g$. 

B. For $y\ra +\infty$ we have 

\[
g(y) = \left(\frac{2}{9}\right)^{\frac{1}{6}}yU\left(\frac{1}{6},\frac{4}{3},\frac{2}{9}y^3\right)\sim \left(\frac{2}{9}\right)^{\frac{1}{6}}y\left(\frac{2}{9}y^3\right)^{-\frac{1}{6}} = y^{\frac{1}{2}},
\]
where we have used the definition of $U$. Thus, we have for $h$
\[
h(x,y) = x^{\frac{1}{6}}g(x^{-\frac{1}{3}}y)\sim x^{\frac{1}{6}}(x^{-\frac{1}{3}}y)^{\frac{1}{2}} = y^{\frac{1}{2}}. 
\]

C. For $y\ra-\infty$ we have 

\begin{align*}
    g(y) &= - \left(\frac{2}{9}\right)^{\frac{1}{6}}\frac{1}{6}yV\left(\frac{1}{6},\frac{4}{3},\frac{2}{9}y^3\right) = - \left(\frac{2}{9}\right)^{\frac{1}{6}}\frac{1}{6}ye^{\frac{2}{9}y^3}U\left(\frac{7}{6},\frac{4}{3},-\frac{2}{9}y^3\right)\\&\sim -  \left(\frac{2}{9}\right)^{\frac{1}{6}}\frac{1}{6}ye^{\frac{2}{9}y^3} \left(-\frac{2}{9}y^3\right)^{-\frac{7}{6}} = \frac{3}{4}(-y)^{-\frac{5}{2}}e^{\frac{2}{9}y^3}. 
\end{align*}
Here we have used the properties of $U$. Thus we have for $h$
\[
h(x,y) = x^{\frac{1}{6}}g(x^{-\frac{1}{3}}y)\sim x^{\frac{1}{6}}\frac{3}{4}(-x^{-\frac{1}{3}}y)^{-\frac{5}{2}}e^{\frac{2}{9}\frac{y^3}{x}} = \frac{3}{4}x(-y)^{-\frac{5}{2}}e^{\frac{2}{9}\frac{y^3}{x}}.
\]

\end{proof}

Next, we establish some useful estimates. In the following we use the shorthand $$q_t(x,y;u,v) :=  p_t(x,y;u,v) - p_t(x,y;u,-v).$$
\begin{lemma}
\label{thm:lemmahelpintegral1}
(a) Define the quadratic forms $Q_1(u,v,z) = 6u^2-6uv+2v^2+2z^2-6uz+2vz$ and $Q_2(u,v,z) = 6u^2-6uv+2v^2+2z^2+6uz-2vz$. Then it follows 
\[
Q_1(u,v,z)\ge c[(z+u)^2+(v-u)^2],
\]
\[
Q_2(u,v,z) \ge   c[(z-u)^2+(v-u)^2].
\]


(b) Define
\[
\tilde\rho(u,v,z) := (u+v)\exp(-(u-v)^2)\left[\exp(-(z-u)^2)+\exp(-(z+u)^2)\right].
\]
Then the function $\R_+\times\R\ni (u,v)\mapsto h(u,v)\sup_{|z|\le\epsilon}\tilde\rho(u,v,z)$ is integrable for any $\epsilon>0$.  

(c) For all $u,z>0$ and $v$ 
\begin{equation*}
p_1(u,-v;0,\pm z)\leq\frac{\sqrt[]{3}}{\pi}\exp\left(-6u^2-v^2+9-z^2\right).
\end{equation*}
(d) For all $u,z>0,v$ and $s\in (0,1)$
\begin{equation*}
|q_{1-s}(u,-v;0,\pm z)|\leq\frac{C}{(1-s)^2}\exp\left(-\frac{6u^2}{(1-s)^3}-\frac{v^2+z^2}{1-s}\right). 
\end{equation*}
\end{lemma}
\begin{proof}
(a) This can be proved through lengthy algebra. Another way forward is to note that $Q_1(\tilde z) = \tilde z^t A\tilde z, \tilde z\in \R^3$ with the positive semidefinite matrix 
\[
A = \begin{pmatrix}
6 & -3 & -3\\
-3 & 2 & 1\\
-3 & 1 & 2
\end{pmatrix}.
\]
$A$ has eigenvalues $0,1,9$ and diagonalization given by 
\[
A = \begin{pmatrix}
1 & 0 & -2\\
1 & -1 & 1\\
1 & 1 & 1
\end{pmatrix}\cdot
\begin{pmatrix}
0 & 0 & 0\\
0 & 1 & 0\\
0 & 0 & 9
\end{pmatrix}\cdot
\begin{pmatrix}
\frac{1}{3} & \frac{1}{3} & \frac{1}{3}\\
0 & -\frac{1}{2} & \frac{1}{2}\\
-\frac{1}{3} & \frac{1}{6} & \frac{1}{6}
\end{pmatrix}.
\]
From here, simple algebra delivers the estimate for $Q_1$. 

Similarly, $Q_2(\tilde z) = \tilde z^t B\tilde z, \tilde z\in \R^3$ with the positive semidefinite matrix
\[
B = \begin{pmatrix}
6 & -3 & 3\\
-3 & 2 & -1\\
3 & -1 & 2
\end{pmatrix}.
\]
$B$ has eigenvalues $0,1,9$ and diagonalization given by 
\[
B = \begin{pmatrix}
-1 & 0 & 2\\
-1 & 1 & -1\\
1 & 1 & 1
\end{pmatrix}\cdot
\begin{pmatrix}
0 & 0 & 0\\
0 & 1 & 0\\
0 & 0 & 9
\end{pmatrix}\cdot
\begin{pmatrix}
-\frac{1}{3} & -\frac{1}{3} & \frac{1}{3}\\
0 & \frac{1}{2} & \frac{1}{2}\\
\frac{1}{3} & -\frac{1}{6} & \frac{1}{6}
\end{pmatrix}.
\]
From here, simple algebra delivers the estimate for $Q_2$. 



(b) Recall that $h(u,v)\le c\alpha(u,v)^{\frac{1}{2}}$. Note also that $u+v\le c(\alpha(u,v)^3+\alpha(u,v))$. Hence we concentrate on the exponential part
\[
\iota:(u,v,z)\mapsto
\exp(-(u-v)^2)\left[\exp(-(z-u)^2)+\exp(-(z+u)^2)\right].
\]

Suppose that $|u-v|\le 1$. Then trivially $u\ge R\implies v\ge R-1$ and $v\ge R\implies u\ge R-1$. Hence this region does not cause problems for integrability. 


Suppose that $u\ge v+1$. In particular, $u\ge v$. Then for every $z\le \epsilon$ it holds $\iota(u,v,z)\le C_{\epsilon}\exp(-\frac{1}{2}(u^2+v^2))$. This holds true both when $v\ge 0$ and when $v<0$. 

Finally, suppose that $v\ge u+1$. In particular, $v>0$. Then $\exp(-(u-v)^2)\le \exp(-v)\exp(u)$, so that altogether, for every $|z|\le\epsilon$, it holds $\iota(u,v,z)\le C_{\epsilon}\exp(-v)\exp(c_{\epsilon}(u - u^2))$. Overall, it follows that the function is integrable over the  region $|u-v|\ge 1$ also. This finishes the proof of part (b).

(c) Recall that 
\begin{align*}
p_1(u,-v;0,z) &= \frac{\sqrt[]{3}}{\pi}\exp\left(-6u^2+6(v-z)-2(z^2-vz+v^2)\right)\\
&= \frac{\sqrt[]{3}}{\pi}\exp\left(-6u^2-2v^2 +(v-z)(6+2z)\right).
\end{align*}
Maximize the parabola $-v^2+(v-z)(6+2z)$ w.r.t $v$ and substitute the value of the maximized function in the exponent. This delivers the statement.

(d) This is straightforward from (c) and the scaling property of $\bar p$. 
\end{proof}

Next, we register a helpful result about densities (see \citet[p.128]{L4}).
\begin{lemma}[Equation for the densities]\label{thm:lemmaeqdensities}
For $x,y,u,v \in \mathbb{R}$ we have
\begin{align*}
p_t(x,y;u,v) = p_t(u, -v; x, -y), \\
\nonumber
\overline{p}_t(x,y;u,v) = \overline{p}_t(u, -v; x, -y).
\end{align*}
\end{lemma}

As a next step we prove a decomposition for $\bar{p}_1(x,y;u,v)$, which is needed in the proof of the Propositions \ref{thm:garbitlike} and \ref{thm:propbarp}. 

\begin{lemma}
\label{thm:decomposition}
It holds true for $x$,$u > 0$ and $y$,$v \in \mathbb{R}$
\begin{align*}
\overline{p}_t(x,y;u,v) &= p_t(x,y;u,v) - p_t(x,y;-u,-v)  \\
                        &- \int_{0}^t \int_{0}^\infty q_{t-s}(u,-v;0,z)|z|q_s(x,y;0,-z) 
\textrm{d}w \textrm{d}r \textrm{d}z \textrm{d}s\\&
+\int_0^1\int_0^{\infty}\int_0^s\int_0^{\infty}zq_{1-s}(u,-v;0,z)f_r(0,-z;0,w)q_{s-r}(x,y;0,-w)dwdrdzds.
\end{align*}

\end{lemma}

\begin{proof}
Recalling Lemma \ref{thm:formulabarp}, we have for $x$,$u > 0$ and $y$,$v \in \mathbb{R}$ that
\begin{equation}\label{eq:firstbarp}
\overline{p}_t(x,y;u,v) = p_t(x,y;u,v) - \int_0^t \int_{-\infty}^0 f_s(x,y;0,z) p_{t-s}(0,z;u,v) \textrm{d}z \textrm{d}s.  
\end{equation}
By $p_t(x,y;u,v) = p_t(u,-v,x,-y)$ (see Lemma \ref{thm:lemmaeqdensities}) and the substitution $z \rightarrow -z$ equals
\begin{equation}\label{eq:secondbarp}
p_t(x,y;u,v) - \int_{0}^t \int_{0}^\infty f_s(x,y;0,-z) p_{t-s}(u,-v;0,z) \textrm{d}z \textrm{d}s. 
\end{equation}
As $q_t(x,y;u,v) := p_t(x,y,u,v) - p_t(x,y;u,-v)$, \eqref{eq:secondbarp} equals
\begin{align}\label{eq:thirdbarp}
& p_t(x,y;u,v) - \int_{0}^t \int_{0}^\infty f_s(x,y;0,-z) (q_{t-s}(u,-v;0,z)+p_{t-s}(u,-v;0,-z)) \textrm{d}z \textrm{d}s. 
\end{align} 
Because every path which goes from $(x,y)$ to $(-u,-v)$ in time $t$ has to pass through $\{0\} \times (-\infty,0]$ at least once, we have using the symmetry of density, that
\begin{align}\label{eq:fourthbarp}
~& p_t(x,y;-u-v) \\
\nonumber
=&\int_{0}^t \int_{0}^\infty f_s(x,y;0,-z) p_{t-s}(0,-z;-u,-v)) \textrm{d}z \textrm{d}s  \\
\nonumber
=&\int_{0}^t \int_{0}^\infty f_s(x,y;0,-z) p_{t-s}(0,z;u,v))  \textrm{d}z \textrm{d}s \\
\nonumber
=&\int_{0}^t \int_{0}^\infty f_s(x,y;0,-z) p_{t-s}(u,-v;0,-z)) \textrm{d}z \textrm{d}s.
\end{align} 
Applying \eqref{eq:fourthbarp} to \eqref{eq:thirdbarp}, we get that $\overline{p}_t(x,y;u,v)$ equals
\begin{align}\label{eq:fifthbarp}
&p_t(x,y;u,v) - p_t(x,y;-u-v) - \int_{0}^t \int_{0}^\infty f_s(x,y;0,-z) q_{t-s}(u,-v;0,z) \textrm{d}z \textrm{d}s.
\end{align}
Now we look closer at the double integral in \eqref{eq:fifthbarp} and apply Lemma \ref{thm:jointdistlemma}:
\begin{align}\label{eq:sixtheq}
\nonumber
&\int_{0}^t \int_{0}^\infty f_s(x,y;0,-z) q_{t-s}(u,-v;0,z) \textrm{d}z \textrm{d}s = \\
& \int_{0}^t \int_{0}^\infty q_{t-s}(u,-v;0,z)|z|\Big[p_s(x,y;0,-z) -  \\
\nonumber
& ~~~~~~~~~ \int_{0}^{s} \int_0^{\infty} f_r(0,-|z|;0,w) p_{s-r}(x,y;0,-w) \textrm{d}w \textrm{d}r \Big] \textrm{d}z \textrm{d}s.  
\end{align}
 Since every path from $(0,-z)$ to $(x,-y)$ hits $\{0\} \times [0, \infty)$ at least once, we have
\begin{align*}
~&\int_{0}^s \int_{0}^\infty f_r(0,-z;0,w) p_{s-r}(x,y;0,-w)) \textrm{d}w \textrm{d}r \\
\nonumber
=&\int_{0}^s \int_{0}^\infty f_r(0,-z;0,w) p_{s-r}(0,w;x,-y)) \textrm{d}w \textrm{d}r \\
\nonumber
=& p_s(0,-z,x,-y) = p_s(x,y;0,z),
\end{align*}
where we have used Markov property.
By this calculation and the definition of $q$ the integrand of \eqref{eq:sixtheq} is

\[
q_{t-s}(u,-v;0,z)|z|\left[q_s(x,y;0,-z)-\int_0^s\int_0^{\infty}f_r(0,-|z|;0,w)q_{s-r}(x,y;0,w)dwdr\right].
\]
This gives the result.
\end{proof}




\section{Proofs of Propositions \ref{thm:garbitlike} and \ref{thm:propbarp}}



\begin{proof}[Proof of Proposition \ref{thm:garbitlike}]
We recall the decomposition of $\bar{p}$ from Lemma \ref{thm:decomposition}. 
\begin{align*}
\overline{p}_t(x,y;u,v) &= p_t(x,y;u,v) - p_t(x,y;-u,-v)  \\
                        &- \int_{0}^t \int_{0}^\infty q_{t-s}(u,-v;0,z)|z|q_s(x,y;0,-z) 
\textrm{d}w \textrm{d}r \textrm{d}z \textrm{d}s\\&
+\int_0^1\int_0^{\infty}\int_0^s\int_0^{\infty}zq_{1-s}(u,-v;0,z)f_r(0,-z;0,w)q_{s-r}(x,y;0,-w)dwdrdzds.
\end{align*}


\paragraph{\textbf{Step 1.}} First we deal with the part 
\[
p_1(x,y;u,v) - p_1(x,y;-u,-v).
\]
Using the definition of $p_1$ this is equal to
\begin{align*}
&p_1(x,y;u,v) - p_1(x,y;-u,-v) \\&= \frac{2\sqrt{3}}{\pi}\exp\left(-6(u^2+x^2)+6(uv-xy)-2(v^2+y^2)\right)\sinh\left(12ux+6(uy-vx)-2vy\right).
\end{align*}
Now recall, that there exists $c>0$ with 
\[
-6u^2\pm 6uv-2v^2\leq -c(u^2+v^2).
\]
It follows 
\[
\exp\left(-6(u^2+x^2)+6(uv-xy)-2(v^2+y^2)\right)\leq \exp\left(-c(u^2+v^2)\right)\exp\left(-c(x^2+y^2)\right).
\]
Here, we have used $x,y>0$. 

Using the estimate $\sinh(x)\leq x\cosh(x)\leq xe^x$ for $x\geq 0$ it follows that there exists some $c'>0$ with  
\[
\sinh\left(12ux+6(uy-vx)-2vy\right)\leq (12|u|+6|v|)(x+y)\exp(c'(|u|+|v|)).
\]
Here we have used that $(x,y)$ is bounded. Recall from Lemma 6 in \citet{dw12}
\[
h(x,y)\geq c\max\{x^{\frac{1}{6}},y^{\frac{1}{2}}\},
\]
with $c>0$ suitable.
In all, we have for suitable $\gamma>0$
\[
|p_1(x,y;u,v) - p_1(x,y;-u,-v)|=o(1) h(x,y)\exp\left(-\gamma(u^2+v^2)\right), \textrm{ as }x,y\searrow 0.
\]
\paragraph{\textbf{Step 2.}}Next we look at 
\begin{equation*}
\frac{1}{h(x,y)}\int_{0}^1\int_0^{\infty}q_{1-s}(u,-v;0,z)|z|q_s(x,y;0,-z)dzds.
\end{equation*}
We first consider the integration domain $0<\epsilon<s<1$ for an $\epsilon\in (0,1)$ fixed throughout. Using the estimates $\sinh(x)\leq |x|\cosh(x)$, Lemma 6 in \citet{dw12}, $\epsilon<s$ and the definition for $p_s$ we have
\begin{align*}
q_s(x,y;0,-z) &= \frac{2\sqrt{3}}{\pi s^2}\exp\left(-\frac{6x^2}{s^3}-\frac{6xy}{s^2}-\frac{2y^2+2z^2}{s}\right)\sinh\left(\frac{6xz}{s^2}+\frac{2yz}{s}\right)\\
&\leq \left(\frac{6xz}{s^2}+\frac{2yz}{s}\right)\left(p_s(x,y;0,z)+p_s(x,y;0,-z)\right)\\
&\leq C_{\epsilon}z(x+y) = C_{\epsilon}zh(x,y)o(1),\textrm{ as }x,y\searrow 0.
\end{align*}
In the last inequality we use extensively \eqref{eq:pt} and the fact that $\alpha(x,y)\le \frac{1}{2}$. For more details, note first that 
\begin{align*}
    p_s(x,y;0,z)+p_s(x,y;0,-z) &= \frac{\sqrt{3}}{\pi s^2}\exp\left(-\frac{6x^2}{s^3}-\frac{6x(z+y)}{s^2}-\frac{2(z^2+zy+y^2)}{s}\right) \\
    +&\frac{\sqrt{3}}{\pi s^2}\exp\left(-\frac{6x^2}{s^3}-\frac{6x(y-z)}{s^2}-\frac{2(z^2-zy+y^2)}{s}\right).\end{align*}
Note also that the quadratic forms $(z,y)\mapsto z^2\pm zy+y^2$ are positive definite. This delivers easily that for $s\in [\epsilon,1], (x,y)\in\R_+\times\R_+$ such that $\alpha(x,y)\le \frac{1}{2}$, the sum $p_s(x,y;0,z)+p_s(x,y;0,-z)$ is bounded by a positive constant that depends on $\epsilon$. 

Now we look to control the integral 
\begin{equation*}
\int_{\epsilon}^1\int_0^{\infty} z^2q_{1-s}(u,-v;0,z)dzds.
\end{equation*}

Use (d) from Lemma \ref{thm:lemmahelpintegral1} to get
\begin{align*}
\int_{\epsilon}^1\int_0^{\infty} z^2q_{1-s}(u,-v;0,z)dzds&\leq C\int_{\epsilon}^1\int_0^{\infty}z^2(1-s)^{-2}\exp\left(-\frac{6u^2+v^2+z^2}{1-s}\right)dzds\\& = C\int_0^{\infty}t^2e^{-t^2}dt \cdot \int_0^{1-\epsilon}\frac{1}{\sqrt{s}}ds\cdot \exp\left(-\frac{6u^2+v^2}{1-\epsilon}\right)
\\&\leq C\exp(-\frac{C'}{1-\epsilon}(u^2+v^2)).
\end{align*}


Next we consider the integration domain $0\leq s\leq \epsilon<1$. We look at 
\begin{equation*}
\frac{1}{h(x,y)}\int_{0}^{\epsilon}\int_0^{\infty}q_{1-s}(u,-v;0,z)|z|q_s(x,y;0,-z)dzds.
\end{equation*}
We perform the change of variables $s\ra x^{\frac{2}{3}}s,z\ra x^{\frac{1}{3}}z$ to get after the scaling property of $q_s$
\begin{align*}
&\int_{0}^{\epsilon}\int_0^{\infty}q_{1-s}(u,-v;0,z)|z|q_s(x,y;0,-z)dzds \\&= \int_{0}^{x^{-\frac{2}{3}}\epsilon}\int_0^{\infty}x^{\frac{4}{3}}q_{1-x^{\frac{2}{3}}s}(u,-v;0,x^{\frac{1}{3}}z)|z|q_{x^{\frac{2}{3}}s}(x,y;0,-x^{\frac{1}{3}}z)dzds\\&=
\int_{0}^{x^{-\frac{2}{3}}\epsilon}\int_0^{\infty}q_{1-x^{\frac{2}{3}}s}(u,-v;0,x^{\frac{1}{3}}z)|z|q_{s}(1,x^{-\frac{1}{3}}y;0,-z)dzds.
\end{align*}
Now recall that $1-x^{\frac{2}{3}}s\in[1-\epsilon,1]$ for $s\in[0,x^{-\frac{2}{3}}\epsilon]$ and use (d) from Lemma \ref{thm:lemmahelpintegral1} to conclude that 
\begin{equation*}
q_{1-x^{\frac{2}{3}}s}(u,-v;0,x^{\frac{1}{3}}z)\leq C_{\epsilon}x^{\frac{1}{3}}z(u+v)\exp(-C(6u^2+v^2+x^{\frac{2}{3}}z^2)).
\end{equation*}
Furthermore, $h(x,y)\geq c|x|^{\frac{1}{6}}$ by Lemma 6 in \citet{dw12}, i.e. one has\\ $x^{\frac{2}{3}} = h(x,y)o(1)$ as $x,y\searrow 0$. Therefore, it suffices to show that the following expression is finite 
\begin{equation*}
\int_{0}^{\infty}\int_0^{\infty}zq_{s}(1,x^{-\frac{1}{3}}y;0,-z)dzds.
\end{equation*}
We use (d) from Lemma \ref{thm:lemmahelpintegral1} to see that it suffices 
\begin{equation*}
\int_{0}^{\infty}\int_0^{\infty}zs^{-2}e^{-\frac{1}{s^3}-\frac{z^2}{s}}dzds<\infty. 
\end{equation*}
This is true. 

These two steps establish the result. 


\paragraph{\textbf{Step 3.}} Finally, we have to look at the following four-fold integral.
\begin{align*}
I(x,y,u,v) = \int_0^1\int_0^{\infty}\int_0^s\int_0^{\infty}zq_{1-s}(u,-v;0,z)f_r(0,-z;0,w)q_{s-r}(x,y;0,w)dwdrdzds.
\end{align*}

Since the argument is particularly lengthy, we split this step into several intermediate steps.

\emph{Step 3-a: preparations.}

We first use the identity
\[
q_s(x,y;u,-v)=-q_s(x,y;u,v)
\]
twice and then substitute $r\ra s-r$ to get 
\begin{align*}
I(x,y,u,v)=\int_0^1\int_0^{\infty}\int_0^s\int_0^{\infty}zq_{1-s}(u,-v;0,-z)f_{s-r}(0,-z;0,w)q_{r}(x,y;0,-w)dwdrdzds.
\end{align*}
According to the formula for $f_r$ from Lemma \ref{thm:jointdistlemma} we have
\begin{equation*}
f_{s-r}(0,-z;0,w) = \frac{3we^{-\frac{2(z^2-zw+w^2)}{s-r}}}{\pi\sqrt[]{2\pi}(s-r)^2}\int_0^{\frac{4zw}{s-r}}\theta^{-\frac{1}{2}}e^{-\frac{3\theta}{2}}d\theta. 
\end{equation*}
Note that $\int_0^{x}\theta^{-\frac{1}{2}}e^{-\frac{3\theta}{2}}d\theta\le 2\sqrt{x}$ for all $x>0$. 
Note also that the quadratic form $(z,w)\mapsto 2(z^2-zw+w^2)$ is positive definite. These two facts deliver the estimate 
\begin{equation}
   \label{eq:helpherenow}
    zf_{s-r}(0,-z;0,w)\le C\frac{(zw)^{\frac{3}{2}}}{(s-r)^{\frac{5}{2}}}\exp(-c\frac{z^2+w^2}{s-r}).
\end{equation}
Replacing \eqref{eq:helpherenow} into $I$, recalling that $q_r(x,y;0-w)>0$ for $x,y,w>0$ and interchanging the integration sequence (the integrand is non-negative) delivers 
\begin{align*}
I(x,y,u,v)\le C\int_0^1\int_0^{\infty}w^{\frac{3}{2}}q_{r}(x,y;0,-w)\left(\int_0^{\infty}\int_r^1\frac{z^{\frac{3}{2}}\exp(-c\frac{z^2}{s-r})}{(s-r)^{\frac{5}{2}}}q_{1-s}(u,-v;0,-z)dsdz\right)dwdr.
\end{align*}

Recalling \eqref{eq:defh}, we focus the analysis in the following on the function 
\begin{align*}
    \psi(u,v,r,t,a,b) &:= \int_a^{b}\int_r^t\frac{z^{\frac{3}{2}}\exp(-c\frac{z^2}{s-r})}{(s-r)^{\frac{5}{2}}}q_{1-s}(u,-v;0,-z)dsdz,\\&\textrm{for } 0<r<t\le 1, u,v>0, 0\le a<b\le \infty.
\end{align*}

\emph{Step 3-b: integrability of $\R_+\times\R\ni (u,v)\mapsto h(u,v)\sup_{r\in [0,1]}\psi(u,v,r,1,\epsilon,\infty)$ for arbitrary $\epsilon>0$.}

Note that $t^{-\frac{5}{2}} e^{-\frac{c}{2} t^{2}} = O(1), t\ra\infty$. We use this fact, part (d) of Lemma \ref{thm:lemmahelpintegral1} together with the substitution $z\mapsto \frac{z}{\sqrt{s-r}}$ and the fact that $s-r\in (0,1)$ to estimate 
\begin{align*}
    \psi(u,v,r,\epsilon,\infty)\le C_{\epsilon}\left(\int_{0}^\infty z^4e^{-\frac{c}{2}z^2}dz\right)\cdot \int_0^1\frac{1}{(1-s)^2}\exp\left(-\frac{6u^2}{(1-s)^3}-\frac{v^2}{1-s}\right)ds. 
\end{align*}
$z\mapsto z^4e^{-\frac{c}{2}z^2}$ is integrable. Suppose that $u\ge R$. Then it follows 
\begin{align}
\label{eq:ja futim kot}
\int_0^1\frac{1}{(1-s)^2}\exp\left(-\frac{6u^2}{(1-s)^3}-\frac{v^2}{1-s}\right)ds &\nonumber\le C\exp(-c(u^2+v^2))\int_0^1\frac{1}{(1-s)^2}\exp\left(-\frac{R^2}{(1-s)^3}\right)ds\\&
\le C_R\exp(-c(u^2+v^2)).
\end{align}
Suppose that $u< R, |v|\ge R$. Then a similar estimate to \eqref{eq:ja futim kot} holds true. This finishes the integrability required.

\emph{Step 3-c: integrability of  $\R_+\times\R \ni(u,v)\mapsto h(u,v)\sup_{r,\epsilon: r+\epsilon<1}\psi(u,v,r+\epsilon,1,0,\epsilon)$.}

We use first a similar estimate as in Step 2 of the proof based on the fact that $\sinh(x)\le|x|\cosh(x)$. Namely, it holds
\begin{align*}
    q_{1-s}(u,-v;0,-z)&\le \frac{C}{(1-s)^2}\exp\left(-\frac{6u^2}{(1-s)^3}+\frac{6uv}{(1-s)^2}-\frac{2v^2+2z^2}{1-s}\right)\\&\times \sinh\left(z(\frac{6u}{(1-s)^2}-\frac{2v}{1-s})\right)\\
    &\le C z\left(\frac{6u}{(1-s)^2}+\frac{2v}{1-s}\right)\exp\left(-\frac{6u^2}{(1-s)^3}+\frac{6uv}{(1-s)^2}-\frac{2v^2+2z^2}{1-s}\right)\\&\times\frac{1}{(1-s)^2}\cosh\left(z(\frac{6u}{(1-s)^2}-\frac{2v}{1-s})\right).\\
    =C & z\rho(u,v,s,z),
\end{align*}
with 
\begin{align*}
\rho(u,v,s,z) &:= \left(\frac{6u}{(1-s)^2}+\frac{2v}{1-s}\right)\exp\left(-\frac{6u^2}{(1-s)^3}+\frac{6uv}{(1-s)^2}-\frac{2v^2+2z^2}{1-s}\right)\\&\times\frac{1}{(1-s)^2}\cosh\left(z(\frac{6u}{(1-s)^2}-\frac{2v}{1-s})\right).
\end{align*}
After straightforward algebra and use of part (a) of Lemma \ref{thm:lemmahelpintegral1} it follows that
\begin{align*}
\rho&(u,v,s,z) \le\bar\rho(u,v,s,z):=\frac{1}{(1-s)^\frac{5}{2}} \left(\frac{6u}{(1-s)^\frac{3}{2}}+\frac{2v}{\sqrt{1-s}}\right)\times\\& \left[\exp\left(-cQ_1\left(\frac{u}{(1-s)^{\frac{3}{2}}},\frac{v}{\sqrt{1-s}},\frac{z}{\sqrt{1-s}}\right)\right)+\exp\left(-cQ_2\left(\frac{u}{(1-s)^{\frac{3}{2}}},\frac{v}{\sqrt{1-s}},\frac{z}{\sqrt{1-s}}\right)\right)\right]. 
\end{align*}
To show integrability of integrability of  $\R_+\times\R \ni(u,v)\mapsto h(u,v)\psi(u,v,r+\epsilon,1,0,\epsilon)$ for $\epsilon>0$ small so that $\epsilon+r<1$, we note first that it is without loss of generality to show it for $r\le\frac{1}{2}$ and $\epsilon\le\frac{1}{4}$. This is because the integrand is non-negative for all $r\in (0,1),\epsilon,u,v>0$. We focus on such values of $r,\epsilon$ in the following. 

We use the substitution $z\mapsto \frac{z}{\sqrt{s-r}}$ to arrive at 
\[
\psi(u,v,r+\epsilon,1,0,\epsilon)\le \int_{r+\epsilon}^1\int_0^{\frac{\epsilon}{\sqrt{s-r}}} z^{\frac{5}{2}}e^{-cz^2}\frac{1}{(s-r)^{\frac{3}{4}}}\bar\rho(u,v,s,z\sqrt{s-r})dzds. 
\]
Note first, that $z\mapsto z^{\frac{5}{2}}e^{-cz^2},z\ge 0$ is bounded,  and that $\frac{1}{(s-r)^{\frac{3}{4}}}\le C_\epsilon$. 

We make use of the substitutions $u\mapsto\frac{u}{(1-s)^{\frac{3}{2}}}, (v,z)\mapsto\frac{(v,z)}{\sqrt{1-s}}$ together with the fact that $z\sqrt{s-r}\le \epsilon$ to see, that to complete Step 3-c, it is sufficient to show that
\[
\int_0^\infty\int_0^\infty h(u,v)\sup_{z\le \epsilon}\bar\rho(u,v,0,z)dudv 
\]
is finite. But this follows from easy substitutions and part (b) of Lemma \ref{thm:lemmahelpintegral1}.

\emph{Step 3-d: integrability of  $\R_+\times\R\times \ni(u,v)\mapsto h(u,v)\sup_{r,\epsilon:r+\epsilon<1}\psi(u,v,r,r+\epsilon,0,\epsilon)$.}

Suppose first that $r> \frac{1}{2}$. Then we are integrating $s$ over a range $[r,r+\epsilon]$, which can be expanded to $[\frac{1}{2},\frac{3}{4})\cup[\frac{3}{4},1]$ (recall that the integrand is non-negative). Step 3-c shows the integrability needed in the range $[\frac{3}{4},1]$. Hence, in the following we work on the integrability of $\R_+\times\R_+\times \ni(u,v)\mapsto h(u,v)\psi(u,v,r,\frac{3}{4},0,\epsilon)$ for $\epsilon>0$ for some $r\le \frac{1}{2}$. Note that $1-s\in [\frac{1}{4},1)$ under these conditions. 

Analogous to the Step 3-c, we use the substitution $z\mapsto \frac{z}{\sqrt{s-r}}$ to arrive at 
\[
\psi(u,v,r,\frac{3}{4},0,\epsilon)\le \int_{r}^\frac{3}{4}\int_0^{\frac{\epsilon}{\sqrt{s-r}}} z^{\frac{5}{2}}e^{-cz^2}\frac{1}{(s-r)^{\frac{3}{4}}}\bar\rho(u,v,s,z\sqrt{s-r})dzds. 
\]
Note that $z\sqrt{s-r}\le \epsilon$. We estimate for $z'\le \epsilon$
\begin{align*}
    \bar\rho&(u,v,s,z')\le C(u+v)\exp\left(-c(\frac{u}{1-s}-v)^2\right)\times\\&\left[\exp\left(-c(z'-\frac{u}{1-s})^2\right)+\exp\left(-c(z'-\frac{u}{1-s})^2\right)\right].
\end{align*}
By steps analogous to the proof of part (b) of Lemma \ref{thm:lemmahelpintegral1}, one arrives at the integrability of
\begin{equation}
    \label{eq:helpherelast}
    \R_+\times\R\ni (u,v)\mapsto h(u,v)\cdot\sup_{s\in [\frac{3}{4},0], z\le\epsilon}\bar\rho(u,v,s,z).
\end{equation}
Finally, recall that $(0,\infty)\ni z\mapsto z^{\frac{5}{2}}e^{-cz^2}$ is integrable, as is $(0,1)\ni t\mapsto \frac{1}{t^{\frac{3}{4}}}$.\footnote{We make the substition $t\mapsto s-r$ in $\int_r^{\frac{3}{4}}\frac{1}{(s-r)^{\frac{3}{4}}}ds$.} This, together with \eqref{eq:helpherelast} finishes the proof of Step 3-d.

\textbf{Step 4.} Step 3 has shown that the four-fold integral from the decomposition in Lemma \ref{thm:decomposition} leads to a bound of the type 

\[
I(x,y,u,v)\le C\int_0^1\int_0^{\infty}w^{\frac{3}{2}}q_{r}(x,y;0,-w)\psi(u,v,r,1,0,\infty)dwdr,
\]
with a function $\R_+\times\R\times(0,1)\ni (u,v,r)\mapsto \psi(u,v,r,1,0,\infty)$ so that 
\begin{equation}
    \label{eq:finally}
    \int_0^{\infty}\int_0^{\infty}h(u,v)\sup_{r\in (0,1)}\psi(u,v,r,1,0,\infty) dudv<\infty.
\end{equation}
Making use of the steps $1$ and $2$, of \eqref{eq:finally}, and of the integral representation of $h$ in \eqref{eq:defh} finishes the proof of the proposition.

\end{proof}

\begin{proof}[\textbf{Proof of Proposition \ref{thm:propbarp}}.]
Due to the scaling properties of the functions $\bar{h},h$ and of the density $p$ it is w.l.o.g. to take $t=1$. \\
We use the equations from Lemma \ref{thm:lemmaeqdensities}. 
Due to the proof of Lemma \ref{thm:garbitlike} it is enough to look at
\begin{align}
\label{eq:limitfinal}
\lim_{x,y\searrow 0}\frac{1}{h(x,y)}\int_0^1\int_0^{\infty}\int_0^s\int_0^{\infty}zq_{1-s}(u,-v;0,z)f_r(0,-z;0,w)q_{s-r}(x,y;0,-w)dwdrdzds.
\end{align}
Lemma \ref{thm:garbitlike} and Lemma \ref{thm:lemmahelpintegral1} ensures a dominating function for the integral whenever $|(u,v)|$ is large enough. Otherwise, for $|(u,v)|$ bounded points (a), (b) from Lemma \ref{thm:lemmahelpintegral1} ensure that there is a dominating function and we can interchange integration and limit in  \eqref{eq:limitfinal}. It is easy to see the convergence is also uniform for $\alpha(u,v)\le R$. 

After transformations similar to the ones in the proof of Lemma \ref{thm:garbitlike} we need to verify the limit (uniformly for $\alpha(u,v)\le R$)
\begin{align*}
& \int_{0}^1 \int_{0}^\infty \int_{0}^{\infty} \int_0^{\infty} \int_0^{4} \lim_{x \downarrow 0, y \downarrow 0}  \frac{1}{h(x,y)}  |z|^{3/2} |w|^{3/2} q_{1-s}(u,-v;0,-z) x^{1/6}q_{r}(1,x^{-1/3}y;0,-w)\\ & ~~~~~\times \frac{3}{\pi \sqrt{2 \pi} (s-x^{2/3}r)^{5/2}}
  e^{-\frac{2}{(s-x^{2/3}r)}(z^2  -|x^{1/3}zw| + x^{2/3}w^2)} 
 \theta^{-1/2} e^{-\frac{3|x^{1/3}zw|\theta}{2(s-x^{2/3}r)}} \textrm{d}\theta \textrm{d}w \textrm{d}r  \textrm{d}z \textrm{d}s = \overline{h}(1,u,-v).
\end{align*} 
First we look at the case $x^{-\frac{1}{3}}y\ra c$.
Using $h(x,y) \sim x^{1/6}g(c)$ the integrand equals after substitutions $r\mapsto x^{\frac{2}{3}}r,z\mapsto x^{\frac{1}{3}z}$
\begin{align*}
 &\lim_{x \downarrow 0, y \downarrow 0}  \frac{1}{g(c) } 
  |z|^{3/2} |w|^{3/2} q_{1-s}(u,-v;0,-z) q_{r}(1,x^{-1/3}y;0,-w) \\
\nonumber
 & ~~~~~ \times  \frac{3}{\pi \sqrt{2 \pi} (s-x^{2/3}r)^{5/2}}  e^{-\frac{2}{(s-x^{2/3}r)}(z^2  -|x^{1/3}zw| + x^{2/3}w^2)} 
 \theta^{-1/2} e^{-\frac{3|x^{1/3}zw|\theta}{2(s-x^{2/3}r)}}.
\end{align*}
Solving the limit, this is equal to 
\begin{align*}
\frac{1}{g(c)} |z|^{3/2} |w|^{3/2} q_{1-s}(u,-v;0,-z) q_{r}(1,c;0,-w) 
 \frac{3}{\pi \sqrt{2 \pi} s^{5/2}}  e^{-\frac{2}{s}z^2} \theta^{-1/2},
\end{align*}
uniformly for $\alpha(u,v)\le R$.

After calculating the $\theta$ integral and sorting the terms, we have the integral
\begin{align*}
\nonumber
~& \frac{1}{g(c)} \int_{0}^1 \int_{0}^\infty \frac{4\sqrt{3}}{\sqrt{2 \pi}} |z|^{3/2} s^{-1/2} q_{1-s}(u,-v,0,-z)   \frac{\sqrt{3}}{\pi s^2} e^{-\frac{2}{s}z^2}  \\
  & ~~~~~ \times  \int_0^\infty \int_0^\infty |w|^{3/2} q_{r}(1,c;0,-w) \textrm{d}w \textrm{d}r \textrm{d}z \textrm{d}s\\
=&  \int_{0}^1 \int_{0}^\infty \frac{4\sqrt{3}}{\sqrt{2 \pi}}  |z|^{3/2} s^{-1/2} q_{1-s}(u,-v,0,-z) p_s(0,z;0,0) \\
=& \overline{h}(1,u,-v),
\end{align*}
where in the last steps we used the definitions of $h,g,p$ and $\overline{h}$.
\\ Finally, we look at the case $x^{-1/3}y \rightarrow \infty$.

Note that $x^{-1/3}y \rightarrow \infty$ for $x,y\searrow 0$ is equivalent to $x = 0, y \downarrow 0$. To see this, one can perform a translation of the start points $(x,y)$. Using (4.5) of \citet{groeneboom} we have
\begin{equation*}
\overline{p}_t(0,y;u,v) \sim  \sqrt{y} \overline{h}(t,u,-v) \textrm{~for~} y \downarrow 0.
 \end{equation*}
 It is easy to see from the proof in \citet{groeneboom} that the convergence is uniform in $(u,v)$ as long as they remain bounded.
 
With Lemma \ref{thm:convbehaviorh} we know that  $\sqrt{y} \sim h(x,y)$ as $x,y \searrow 0$ and $x^{-1/3}y \rightarrow \infty$. This concludes the proof.
\end{proof}

\paragraph{\textbf{Disclaimer.}} Michael B\"ar worked on this project in his personal capacity. Any opinions expressed in this article are his own and do not reflect the views of msg systems AG.

\bibliographystyle{apalike}
\bibliography{references.bib}

\end{document}